\documentclass[a4paper, 11pt]{amsart}
\usepackage[utf8]{inputenc}
\usepackage[dvips]{graphicx}
\usepackage{amsmath,amssymb,color,enumitem,graphics,hyperref,latexsym,pgfplots,tikz,vmargin}
\hypersetup{
	unicode=false,          
	pdftoolbar=true,        
	pdfmenubar=true,        
	pdffitwindow=false,     
	pdfstartview={FitH},    
	pdftitle={The Nyquist sampling rate for spiraling curves},    
	pdfauthor={Ph. Jaming, F. Negreira, J.L. Romero},     
	pdfsubject={Sampling Theory},   
	pdfkeywords={Sampling Theory, nyquist rate, undersampling, approximate aliasing}, 
	pdfnewwindow=true,      
	colorlinks=true,       
	linkcolor=black,          
	citecolor=black,        
	filecolor=magenta,      
	urlcolor=cyan           
}

\def\aset{{\Lambda}}
\def\BV{\mathrm{BV}}
\def\C{{\mathbb{C}}}
\def\Cc{{\mathcal{C}}}
\def\curve{{\gamma}}
\def\d{\,{\mathrm{d}}}
\def\diam{{\mathrm{diam}}}
\def\dist{{\mathrm{dist}}}
\def\dirac{{\delta}}
\def\eps{{\varepsilon}}
\def\ffi{{\varphi}}
\def\ess{\mathrm{ess\,sup}}
\def\Ff{{\mathcal{F}}}
\def\gap{{\mathrm{gap}}}
\def\Hh{{\mathcal{H}}}
\def\N{{\mathbb{N}}}
\def\R{{\mathbb{R}}}
\def\Rr{{\mathcal{R}}}
\def\sep{{\eta}}

\def\supp{{\mathrm{supp}}}
\def\traj{{\Gamma}}
\def\var{{\mathrm{var}}}
\def\Z{{\mathbb{Z}}}

\newcommand{\weakconv}{\xrightarrow{w}}
\newcommand{\norm}[1]{{\|{#1}\|}}
\newcommand{\abs}[1]{{|{#1}|}}
\newcommand{\scal}[1]{{\langle{#1}\rangle}}

\newtheorem{lemma}{Lemma}[section]
\newtheorem{proposition}[lemma]{Proposition}
\newtheorem{theorem}[lemma]{Theorem}
\newtheorem{corollary}[lemma]{Corollary}

\theoremstyle{definition}
\newtheorem{definition}[lemma]{Definition}

\theoremstyle{remark}
\newtheorem{remark}[lemma]{Remark}

\title{The Nyquist sampling rate for spiraling curves}
\author[Ph. Jaming, F. Negreira \& J. L. Romero]{Philippe Jaming, Felipe Negreira \& José Luis Romero}

\address{Univ. Bordeaux, IMB, UMR 5251, F-33400 Talence, France. CNRS, IMB, UMR 5251, F-33400 Talence, France.}
\email{philippe.jaming@math.u-bordeaux.fr}
\email{felipe.negreira@math.u-bordeaux.fr}

\address{Faculty of Mathematics, University of Vienna, Oskar-Morgenstern-Platz 1, 1090 Vienna, Austria\\and\\Acoustics Research Institute, Austrian Academy of Sciences, Wohl\-leben\-gasse 12-14, 1040 Vienna, Austria}
\email{jose.luis.romero@univie.ac.at, jlromero@kfs.oeaw.ac.at}

\begin{document}
\begin{abstract}
We consider the problem of reconstructing a compactly supported function from samples of its Fourier transform taken along a spiral. We determine the Nyquist sampling rate in terms of the density of the spiral and show that, below this rate, spirals suffer from an approximate form of aliasing. This sets a limit to the amount of undersampling that compressible signals admit when sampled along spirals. More precisely, we derive a lower bound on the condition number for the reconstruction of functions of bounded variation, and for functions that are sparse in the Haar wavelet basis. 
\end{abstract}

\maketitle

\section{Introduction}

\subsection{The mobile sampling problem}

In this article, we consider the reconstruction of a compactly supported function from samples of its Fourier transform taken along certain curves, that we call spiraling. This problem is relevant, for example, in magnetic resonance imaging (MRI), where the anatomy and physiology of a person are captured by moving sensors.

The Fourier sampling problem is equivalent to the sampling problem for bandlimited functions - that is, functions whose Fourier transform are supported on a given compact set. The most classical setting concerns functions of one real variable with Fourier transform supported on the unit interval $[-1/2,1/2]$, and sampled on a grid $\sep\mathbb{Z}$, with $\sep>0$. The sampling rate $\sep$ determines whether every bandlimited function can be reconstructed from its samples: reconstruction fails if $\sep>1$ and succeeds if $\sep\leqslant1$ \cite{unser2000sampling}. The transition value $\sep=1$ is known as the Nyquist sampling rate, and it is the benchmark for all sampling schemes: modern sampling strategies that exploit the particular structure of a certain class of signals are praised because they achieve sub-Nyquist sampling rates.

The sampling theory for bandlimited functions extends to high dimension and irregular sampling geometries \cite{B1,DS,landau1967sampling}, and it is instrumental in the analysis of sampling schemes arising from continuous curves \cite{BW}. The key notion is the Beurling density of a set, which measures the average number of samples per unit volume.

Beurling's density, however, does not properly reflect the acquisition cost when samples are taken along continuous trajectories. In this case, a more relevant metric is the average {\it length} covered by a curve, as a proxy for scanning times \cite{MR3579696,MR3267158,singh2006active,uv2,uv1}. For example, when sampling a function bandlimited to a compact set $\Omega\subset\R^2$ along equispaced parallel lines with direction $\vec{v}\in\mathbb{S}^1$,
\begin{equation*}
L_{\vec{v},\sep}=\{t\vec{v}+\sep k\vec{v}_\perp:t\in\R, k \in \Z\},
\end{equation*}
the critical sampling rate is dictated by the separation between lines $\sep>0$, and by the measure of the maximal cross section of $\Omega$ by hyperplanes perpendicular to $\vec{v}$ \cite{uv2,uv1}. With the introduction of an adequate notion of {\it path-density}, similar results hold also for arbitrary families of parallel lines \cite{gruv15}.

The analysis of general sampling trajectories in terms of length and density is very subtle and challenging, 
and little can be said in full generality \cite{gruv15}. Nevertheless, a solution to the mobile sampling 
problem is expected to be possible for concrete parametric families of curves.

\subsection{Necessary and sufficient recovery guarantees for spiraling curves}
A first contribution of this article is to give sharp necessary and sufficient conditions for Fourier sampling for a parametric family of curves, that we call spiraling. The main examples of these curves are the Archimedes spiral
\begin{equation}
\label{spiral}A^\sep:=\{(\sep\theta\cos2\pi\theta,\sep\theta\sin2\pi\theta):\theta\geqslant0\}
\end{equation}
and the collection of concentric circles
\begin{equation}
\label{circles}O^\sep:=\{(x,y):x^2+y^2=\sep^2k^2, k\in\N\},
\end{equation}
see Figure \ref{spiralandcircle}.
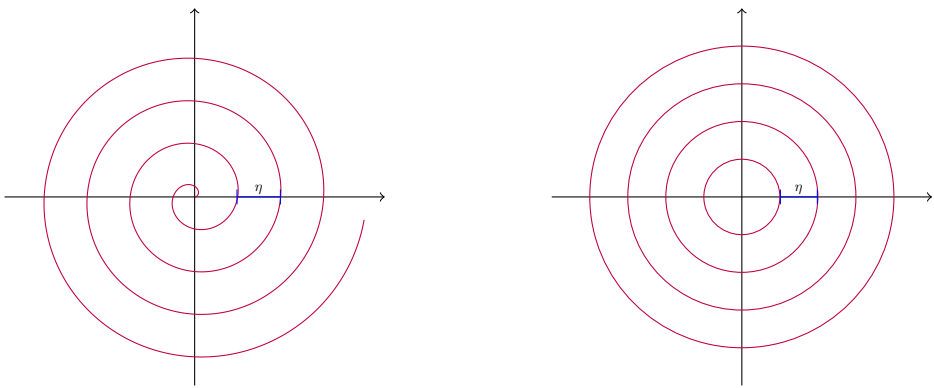
\begin{figure}[ht]
\begin{tikzpicture}[scale=0.5, every node/.style={scale=0.5}]
\draw[->](-5,0)--(5,0);
\draw[->](0,-5)--(0,5);
\draw[purple,domain=0:25,variable=\t,smooth,samples=500]plot({\t r}: {0.18*\t});
\draw[|-|,blue](1.1,0)--(2.27,0);
\node at (1.685,0.2) {$\sep$};
\end{tikzpicture}
\hspace*{5em}
\begin{tikzpicture}[scale=0.5, every node/.style={scale=0.5}]
\draw[->](-5,0)--(5,0);
\draw[->](0,-5)--(0,5);
\draw[purple](0,0)circle(1);
\draw[purple](0,0)circle(2);
\draw[purple](0,0)circle(3);
\draw[purple](0,0)circle(4);
\draw[|-|,blue](1,0)--(2,0);
\node at (1.5,0.2) {$\sep$};
\end{tikzpicture}
\caption{Archimedes spiral (left) and concentric circles (right) with separation $\sep$.}
\label{spiralandcircle}
\end{figure}

We identify the precise Nyquist rate of these curves in terms of the density parameter $\sep$. To be specific, we say that $\traj$, the image of a curve, is a {\it Fourier sampling trajectory} for $\Omega$ - or a {\it sampling trajectory} for the Paley-Wiener space $PW^2(\Omega)$ - if the following continuous sampling inequality holds:
\begin{equation}
\label{eq_cont_samp}A\norm{f}^2_2\leqslant\int_{\traj}\abs{\widehat{f}(\xi)}^2\d\Hh^1(\xi)\leqslant B\norm{f}_2^2,\qquad f\in L^2(\Omega),
\end{equation}
where $A,B>0$ are stability constants, and $\Hh^1$ is the one dimensional Hausdorff (length) measure (see \cite{MR3528398,MR1676043}). Equivalently, $\traj$ is a sampling trajectory if it contains a discrete sampling set (see Section \ref{sec_overview}).

Our first result reads as follows.

\newtheorem*{thmconcrete}{\bf Theorem A}
\begin{thmconcrete}{\em 
Let $\Omega\subset\R^2$ be a convex centered symmetric body.
\begin{enumerate}
\item If $\diam(\Omega)\sep<1$, then the Archimedes spiral $A^\sep$ and the collection of concentric circles $O^\sep$ are sampling trajectories for $PW^2(\Omega)$.

\item If $\diam(\Omega)\sep>1$, then neither the Archimedes spiral $A^\sep$ nor the collection of concentric circles $O^\sep$ are sampling trajectories for $PW^2(\Omega)$.
\end{enumerate}}
\end{thmconcrete}
Part $(i)$ in Theorem A is due to Benedetto and Wu in the context of pointwise sampling \cite{BW}. Our contribution is mainly in $(ii)$.

\subsection{Compressible signals and sampling below the Nyquist rate}
Having identified the Nyquist rate of spiraling curves, we look into undersampling. Modern sampling schemes exploit the fact that many signals of interest are highly compressible, and this information is leveraged to sample below the Nyquist rate. For example, functions defined on the unit square, and obeying a variation bound
\begin{equation*}
\Ff(W):=\{f\in L^2([-1/2,1/2]^2):\var(f)\leqslant W\}
\end{equation*}
are compactly represented in a wavelet basis. Here, the resolution parameter $W$ essentially controls the 
number of active wavelet coefficients \cite{MR1738406, MR2241189}.
\footnote{See
\eqref{eq_def_var} for the definition of $\var(f)$.}

The stability of sampling schemes restricted to such signals is expressed by the inverse condition number \footnote{In standard terminology, the condition number of the sampling problem $f|\traj \mapsto f$ is related to the {\it reciprocal} of this quantity.}
\begin{equation*}
\inf\{\norm{\hat{f}-\hat{g}}_{L^2(\mu_\traj)}:f,g\in\Ff(W)\text{, }\norm{f-g}_2=\eps\},
\end{equation*}
where $\mu_\traj := \Hh^1|_{\traj}$ is the arc-measure, which measures the extent to which the distance between two signals is exhibited by their Fourier samples. Since, $\eps^{-1}\Ff(W)=\Ff(\eps^{-1} W)$, and
$\Ff(W)\subset\Ff(W)-\Ff(W)\subset\Ff(2W)$, the analysis of the condition number for small $\eps$ reduces to the large $W$ asymptotics of the {\it stability margin}:
\begin{equation*}
A(\traj,\Ff(W)):=\inf\{\norm{\hat{f}}_{L^2(\mu_\traj)}:\norm{f}_2=1\text{, }f\in \Ff(W)\}.
\end{equation*}
According to Theorem A, the critical value for the reconstruction of functions defined on the unit square with either $A^\sep$ or $O^\sep$ is $\sep={\sqrt{2}}/{2}$. We consider spirals with density slightly under the critical value and prove the following.

\newtheorem*{thmbv}{\bf Theorem B}
\begin{thmbv}{\em
Let $\sep=(1+\eps)\sqrt{2}/2$ with $\eps\in(0,1)$, and $\traj=A^\sep$ or $\traj=O^\sep$. Then for $W>0$,
\begin{equation*}
A(\traj,\Ff(W))\leqslant C(\eps W)^{-1/2}(\ln^2(\eps W)+1),
\end{equation*}
where $C>0$ is a universal constant.}
\end{thmbv}

Theorem B thus sets a limit to the capacity of spirals to acquire all compressible signals below the Nyquist rate. Informally, it says that when undersampling by a small factor $(1-\eps)$, one can only recover functions up to resolution $W\approx\eps^{-1}$ with a stable condition number.

A variant of Theorem B can be formulated in terms of the Haar wavelet. Let $\Sigma_{N,J}$ be the class of functions on $[-1/2,1/2]^2$ with $N$ non-zero Haar coefficients, all of them taken with scale at most $J$. We have the following estimate.

\newtheorem*{thmwav}{\bf Theorem C}
\begin{thmwav}{\em
Let $W\geqslant1$, $\sep=(1+\eps)\sqrt{2}/2$ with $\eps\in(0,1)$ and $\traj=A^\sep$ or $\traj=O^\sep$. Then for $N\geqslant1$,
\begin{equation*}
A(\traj,\Sigma_{N,J})\leqslant CN^{-1/6}\eps^{-1}\ln^4(C N),
\end{equation*}
where $J=C\ln(\eps^{-1}N)$ and $C>0$ is a universal constant.}
\end{thmwav}
Informally, Theorem C says that when undersampling by a small factor $(1-\eps)$, one can recover at most $N\approx\eps^{-6}$ Haar coefficients with a stable condition number.

Theorem C complements related results that limit the wavelet-sparsity of discrete signals that can be sampled on unions of parallel lines \cite{MR3480103}. Let us mention that the sparsity model $\Sigma_{N,J}$ is rather crude. Modern results in sparse recovery exploit the fine multiscale structure of the wavelet coefficients of natural signals \cite{MR3382105}.

The same fundamental stability restrictions expressed by Theorems C and D also apply to \emph{any posssible discretization of the continuous sampling trajectories} -- see Theorem D below.

\subsection{Overview}\label{sec_overview}
The problem of sampling the Fourier transform of a compactly supported function is equivalent to the sampling problem for the Paley-Wiener space of bandlimited functions. We make essential use of Beurling's sampling theory. The sufficient sampling condition in Theorem A follows from Beurling's gap covering Theorem \cite{B1}, as done in \cite{BW}. The necessary condition in Theorem 
A is more challenging: little of the ample literature on necessary conditions for sampling \cite{MR2224392,MR3742438,MR3050313,MR3796637,MR2231911,MR1744572,landau1967sampling,MR0222554,MR2929058} is applicable to sampling on curves, because the new relevant metric (length) is fundamentally different from the one corresponding to pointwise sampling (cardinality) \cite{gruv15}.

Our main results show that spirals behave qualitatively like unions of parallel lines. While the analysis of sampling on parallel lines is based on periodization arguments \cite{uv2, uv1} and cross sections of the Fourier spectrum \cite{gruv15}, in order to treat spirals, we develop approximate versions of those tools. The main observation is that an adequate sequence of translates of a spiral converges to a union of parallel lines, locally in the Hausdorff metric. This allows us to apply Beurling's characterization of sampling in terms of {\it weak limits} \cite{B1,B2,MR2040080}.

In order to apply weak-limit techniques to curves, we first need to connect pointwise and continuous sampling. We provide a variation of a result from Ortega-Cerdà on {\it sampling measures} \cite{MR1676043}, and show that, under mild regularity assumptions, the continuous sampling inequality \eqref{eq_cont_samp} is equivalent to the existence of a sampling set contained in the sampling trajectory $\traj$ - see Section \ref{sec_traj} for precise definitions.
\newtheorem*{thmcontdiscrete}{\bf Theorem D}
\begin{thmcontdiscrete}{\em 
Let $\traj\subset\R^d$ be a regular trajectory, and $\Omega\subset\R^d$ bounded with positive measure. Then $\traj$ is a sampling trajectory for $PW^2(\Omega)$ if and only if there exists a discrete 
set $\aset\subset\traj$ that is a sampling set for $PW^2(\Omega)$.   }
\end{thmcontdiscrete}
Theorem D shows that two common formulations of the mobile sampling problem are equivalent \cite{uv2,uv1}. As a further consequence of Theorem D, the sampling relation \eqref{eq_cont_samp} also expresses the stability of a vast collection of {\it sampling schemes}, where functions are sampled on finite portions of the sampling trajectory $\traj$, and are reconstructed within a precisely described numerical accuracy \cite{MR3254605,adgaro16,MR2231856,MR2286012}, and leads to well-understood truncation errors \cite{MR2231911,MR1744572} and implementation strategies \cite{MR3209727,MR1613711,MR3398944,MR2679810}. Moreover, any sufficiently dense set $\aset\subset\traj$ is an adequate discretization - cf. Remark \ref{rem_suf_dense}.

As a second step we show that spirals suffer from {\it approximate aliasing}. Aliasing is the name given in signal processing to the artifacts produced by sampling on a lattice below the Nyquist rate. Aliasing is also the most obvious obstruction to subsampling compressible signals, and, heuristically, the success of sub-Nyquist sampling schemes relies on the fact that they avoid regular patterns \cite{MR2228740,lustig2008compressed}. We quantify the rate of converge of spirals to parallel lines and derive approximate aliasing for the curves $A^\sep$ \eqref{spiral} and $O^\sep$ \eqref{circles}. As a consequence we obtain an upper bound on the stability margin for the reconstruction of functions of bounded variation (Theorem B) or of functions that have few active Haar coefficients (Theorem C). These results underscore the need for a certain level of randomness in structured sampling \cite{boyer2017compressed} and for refined multiscale models \cite{MR3615409,MR3382105} that apply to generic signals.

This article is organized as follows: in Section 2 we review standard notions from sampling theory and provide preliminary results; in Section 3 we derive the characterization of sampling trajectories (Theorem D) and introduce spiraling curves; in Section 4 we give necessary and sufficient conditions for sampling on spiraling curves, and prove Theorem A; in Section 5 quantify the rate of convergence of spirals to collections of parallel lines and explore consequences on approximate aliasing; and in Section 6 we prove the results on sub-Nyquist sampling, Theorems B and C.

\section{Preliminaries on pointwise sampling}\label{sec_pre}
\subsection{Notation}
Throughout this paper, we will adopt the following notation: on $\R^d$, $\abs{\cdot}$ and $\scal{\cdot,\cdot}$ are the usual Euclidean norm and scalar product. For sets we will use $\diam E=\sup_{x,y\in E}\abs{x-y}$, $\dist(E,F)=\inf_{x\in E,y\in F}\abs{x-y}$. The balls are denoted by $B_r(x):=\{y\in\R^d:\abs{y-x}<r\}$.
For points on $\R^2$ we will use the notation $x=(x_1,x_2)$, and additionally, vectors on $\mathbb{S}^1$ will be written as $\vec{d}$, $\vec{l}$, and if $\vec{d}=(x_1,x_2)$ we set $\vec{d}_\perp=(-x_2,x_1)$. Clockwise rotations in $\R^2$ will be denoted by $\Rr_{2\pi\theta}$ with $\theta\in[0,1)$ and where $2\pi\theta$ is the angle of rotation. Unless otherwise stated, measures on $\mathbb{R}^d$ are assumed to take values in $[0,+\infty]$.

Throughout the paper, for $A,B\in\R$, $A\lesssim B$ means that there exists a constant $C>0$ independent from $A$ and $B$ such that $A\leqslant CB$. For functions (or measures) $f,g$, $f\lesssim g$ means that  $f(x)\leqslant Cg(x)$ for all $x$ where $f$ and $g$ are defined. Further, $A\asymp B$, means $A\lesssim B\lesssim A$ (and the same for functions or measures). When we want to particularly stress the dependence of the implicit constant $C$ on other factors we may write instead $C_p,C_\curve,C_{\Omega,p},\dots$. 

\subsection{Convex bodies}
A set $\Omega\subset\R^d$ is called a {\it convex body} if it is convex, compact and has non-empty interior. A convex body is called {\it centered} if $0\in\Omega^\circ$ and {\it symmetric} if $\Omega=-\Omega$. We will frequently use the fact that, for a convex centered symmetric body $\Omega$,
\begin{equation}
\label{body}\Omega\subset(1+\eps)\Omega^\circ\text{ and }(1-\eps)\Omega\subset\Omega^\circ,\quad\eps\in(0,1).
\end{equation}

\subsection{Paley-Wiener spaces}
Let us begin by recalling the definition of the standard function spaces involved in sampling theorems. Here and thereafter we will consider the normalized version of the Fourier transform:
\begin{equation*}
\hat{f}(\xi)=\int_{\R^d}e^{-2\pi i\xi\cdot x}f(x)\d x
\end{equation*}
for $f:\R^d\to\R$ integrable. The Fourier transform is then extended to Schwartz distributions in the usual way.

\begin{definition}
Let $\Omega\subset\R^d$ be a compact set of positive measure and $1\leqslant p\leqslant\infty$. We define $PW^p(\Omega)$ as the subset of $L^p(\R^d)$ consisting of Fourier transforms of distributions supported in $\Omega$.
\end{definition}
The classical {\it Paley-Wiener space} corresponds to $p=2$, while $p=\infty$ yields the {\it Bernstein space}. This latter space models possibly non-decaying bandlimited signals, although it has some disadvantages for signal processing, such as lack of invariance under the Hilbert transform. (A remedy to some of these obstacles has been proposed in \cite{MR3223467}.)

Functions in Bernstein spaces are {\it entire functions of exponential type},
as follows for example from the Paley-Wiener Theorem, \cite[Theorem 2.18]{MR591684} \cite[Theorem A]{OU}:
\begin{proposition}\label{PW}
Given a convex centered symmetric body $\Omega\subset\R^d$ and $1\leqslant p\leqslant\infty$, there exists a constant $c>0$ depending only on $\Omega$ such that every $f\in PW^p(\Omega)$ can be extended to an entire function with
\begin{equation*}
\abs{f(x+iy)}\lesssim e^{c\abs{y}},\quad x,y\in\R^d.
\end{equation*}
\end{proposition}
Functions of exponential type enjoy the following norm control of their analytic extensions on horizontal lines.
\begin{proposition}\label{PPL}
Let $f$ be an entire function in $\C^d$ with $\abs{f(x+iy)}\leqslant Me^{c\abs{y}}$ for all $x,y\in\R^d$ and where $M,c>0$ are constants. If $f(x)\in L^p(\R^d)$ with $1\leqslant p\leqslant\infty$ then for all $y\in\R^d$,
\begin{equation*}
\left(
\int_{\R^d}\abs{f(x+iy)}^p\d x
\right)^{1/p}
\leqslant e^{c\abs{y}} \left(\int_{\R^d}\abs{f(x)}^p\d x\right)^{1/p}
\end{equation*}
with the usual modifications when $p=\infty$.
\end{proposition}
For $1\leqslant p<\infty$, Proposition \ref{PPL} is referred as the Plancharel-Pólya inequality and for
$p=\infty$ as the Phragmén–Lindelöf principle; see, e.g., \cite[Theorem 2.11 and 2.16]{MR591684}.

\subsection{Sampling sets and Beurling's gap theorem}
A set $\aset\subset\R^d$ is said to be {\it sampling} for $PW^p(\Omega)$ if
\begin{equation*}
\norm{f}_p\asymp\norm{f}_{\ell^p(\aset)},\quad f\in PW^p(\Omega),
\end{equation*}
where $\norm{f}_{\ell^p(\aset)}=\left(\sum_{\lambda\in\aset}\abs{f(\lambda)}^p\right)^{1/p}$, if $p<\infty$, and $\norm{f}_{\ell^\infty(\aset)}=\sup_{\lambda\in\aset}\abs{f(\lambda)}$.

A set $\aset\subset\R^d$ is {\it separated} if its {\it separation}
\begin{equation*}
\inf_{\lambda,\lambda'\in\aset,\lambda\neq\lambda'}\abs{\lambda-\lambda'}
\end{equation*}
is positive, and it is {\it relatively dense} if its {\it gap} (or hole)
\begin{equation*}
\gap(\aset):=\sup_{x\in\R^d}\inf_{\lambda\in\aset}\abs{x-\lambda}
\end{equation*}
is finite.

The most effective sufficient condition for sampling bandlimited functions in high dimension is formulated in terms of gaps, and is due to Beurling \cite{B1,B2} - see also \cite{BW,OU}.

\begin{theorem}\label{sufficientdense}
Let $\aset\subset\R^d$ and $R>0$. If $\gap(\aset)<1/(2R)$ then $\aset$ is a sampling set for $PW^\infty(\bar{B}_{R/2}(0))$.
\end{theorem}

The value $1/(2R)$ in Theorem \ref{sufficientdense} is critical in the sense that there exists a set $\aset$ with $\gap(\aset)=1/(2R)$ that is not sampling for the spectrum $\bar{B}_{R/2}(0)$. On the other hand, as examples of Theorem \ref{sufficientdense} we note that
\begin{equation}
\label{eq_calc_gaps}\gap(A^\sep)=\gap(O^\sep)=\sep/2,
\end{equation}
see e.g. \cite[Example 2]{BW}, and hence these sets are sampling for $PW^\infty(\bar{B}_{R/2}(0))$ whenever $\sep R<1$.

While Theorem \ref{sufficientdense} applies to arbitrary sets, sometimes it is convenient to work with separated sets. We state without proof the following elementary lemma.

\begin{lemma}\label{lemma_sep_gap}
Let $\aset\subset\R^d$, and $R>\gap(\aset)$. Then there exists a separated set $\aset'\subset\aset$ such that $\gap(\aset')\leqslant R$.
\end{lemma}

\subsection{Sampling with different norms}
The following classical result shows that the sampling problems associated with $PW^2$ and $PW^\infty$ are almost equivalent. See \cite[Theorem 2.1]{OU} for a simple proof.
\begin{theorem}\label{shrink}
Let $\aset\subset\R^d$ be a separated set, $\Omega\subset\R^d$ a compact set of positive measure and $\eps>0$.
\begin{enumerate}
\item If $\aset$ is sampling for $PW^2(\Omega+\bar{B}_\eps(0))$ then it is sampling for $PW^\infty(\Omega)$.

\item If $\aset$ is sampling for $PW^\infty(\Omega+\bar{B}_\eps(0))$ then it is sampling for $PW^2(\Omega)$.
\end{enumerate}
\end{theorem}
As an application, we obtain the following corollary for convex bodies.
\begin{corollary}\label{bodyshrink}
Let $\aset\subset\R^d$ be a separated set, $\Omega\subset\R^d$ a convex centered symmetric body and $\eps\in(0,1)$. 
\begin{enumerate}
\item If $\aset$ is sampling for $PW^2(\Omega)$, then it is sampling for $PW^\infty((1-\eps)\Omega)$. 

\item If $\aset$ is sampling for $PW^\infty(\Omega)$, then it is sampling for $PW^2((1-\eps)\Omega)$. 
\end{enumerate}
\end{corollary}
\begin{proof}
By \eqref{body}, $(1-\eps)\Omega\subset\Omega^\circ$. Since $\Omega$ is compact, this implies that \[\dist((1-\eps)\Omega,(\Omega^\circ)^c)>\eps',\] and, therefore, $(1-\eps)\Omega+\bar{B}_{\eps'}(0)\subset\Omega$. The conclusions now follow from Theorem \ref{shrink}.
\end{proof}

\subsection{Characterization of sampling with weak limits}
\begin{definition}
A set $\aset\subset\R^d$ is called a {\it set of uniqueness} for $PW^p(\Omega)$ if $f\in PW^p(\Omega)$ with $f|_\aset=0$ implies $f\equiv0$.
\end{definition}

Sampling sets are sets of uniqueness. The converse is not true, but it is a remarkable insight due to Beurling, that it is still possible to characterize a sampling set through the uniqueness of what are called its {\it weak limits}.

\begin{definition}\label{weak}
Let $\aset\subset\R^d$ be a closed set and let $\{\aset_n\}_{n \geqslant1}$ be a sequence of closed sets. Then we say that $\{\aset_n\}_{n \geqslant1}$ converges weakly to $\aset$ if for all $R,\eps>0$ there exist $n_{R,\eps}$ such that
\begin{gather*}
\aset_n\cap(-R,R)^d\subset\aset+B_\eps(0),
\\
\aset\cap(-R,R)^d\subset\aset_n+B_\eps(0),
\end{gather*}
hold for all $n\geqslant n_{R,\eps}$. In this case we write $\aset_n\weakconv\aset$.
\end{definition}
Weak convergence can also be formulated in terms of the \emph{Hausdorff distance} between two sets $X,Y \subseteq \R^d$:
\begin{align*}
d_H(X,Y) := \inf \left\{\varepsilon>0: X \subseteq Y+ B_\varepsilon(0),
Y \subseteq X+ B_\varepsilon(0) \right\}.
\end{align*}
Indeed, $\aset_n\weakconv\aset$ if and only if
\begin{align}
\label{eq_weak_conv}
d_H \big( (\Lambda_n \cap \bar{B}_R(x)) \cup \partial \bar{B}_R(x))
, (\Lambda  \cap \bar{B}_R(x) ) \cup \partial \bar{B}_R(x))
\big) \to 0, \qquad \forall x\in \R^d , R>0 \,  . 
\end{align}
(To appreciate the role of the boundary of the ball in the last equation, consider the following example in dimension $d=1$: $\Lambda_n:=\{1+1/n\}$, $\Lambda:=\{1\}$ and $B_R(x)=[0,1]$.) 

The type of weak limits that are needed to characterize sampling sets are those resulting from translates. We will denote the set of weak limits of translates of $\aset$ by $W(\aset)$. Hence, $\aset'\in W(\aset)$ if and only if there exists a sequence $\{x_n\}_{n\in\N}\subset\R^d$ such that $\aset+x_n\weakconv\aset'$.

\begin{theorem}[Beurling, {\cite[Theorem 3, pg. 345]{B2}}]\label{Beurling}
Let $\Omega\subset\R^d$ be a convex centered symmetric body and let $\aset\subset\R^d$. Then $\aset$ is a sampling set for $PW^\infty(\Omega)$ if and only if for all weak limits $\aset'\in W(\aset)$, $\aset'$ is a set of uniqueness for $PW^\infty(\Omega)$.
\end{theorem}

We will use the following compactness result. See, e.g., \cite[Section 4]{grorro15} for proofs.
\begin{lemma}\label{lemma_wc}
Let $\aset\subset\R^d$ be a separated set and $\{x_n\}_{n\geqslant1} \subset\R^d$. Then there exist a subsequence $\{x_{n_k}\}_{k\geqslant1}$ and a separated set $\aset'\subset\R^d$ such that $\aset+x_{n_k}\weakconv\aset'$.
\end{lemma}

\subsection{Sampling measures}
A Borel measure $\mu$ on $\R^d$ is said to be {\it sampling} for $PW^p(\Omega)$ if
\begin{equation*}
\norm{f}_p\asymp\norm{f}_{L^p(\mu)},\qquad f\in PW^p(\Omega),
\end{equation*}
where $\norm{f}_{L^p(\mu)}=(\int_{\R^d}\abs{f(x)}^p\d\mu)^{1/p}$ if $1\leqslant p<\infty$, and $\norm{f}_{L^\infty(\mu)}=\ess_\mu\abs{f}$.

Viewed in this way, a set $\aset$ is sampling when the associated point measure $\delta_\aset:=\sum_{\lambda\in\aset}\delta_\lambda$ is a sampling measure. Let us first notice that sampling measures are uniformly bounded:

\begin{lemma}\label{supr}
Let $\Omega\subset\R^d$ be a convex centered symmetric body and $1\leqslant p<\infty$. Let $\mu$ be a sampling measure for $PW^p(\Omega)$. Then, for all $R>0$, $\sup_{x\in\R^d}\mu(B_R(x))<\infty$.
\end{lemma}

\begin{proof}
First note that it is enough to construct $f_0\in PW^p(\Omega)$ such that $\abs{f_0}\gtrsim\chi_{(-R,R)^d}$, where the implied constant may depend on $R$. Indeed, once $f_0$ is given, we define $f_x(t):=f_0(t-x)$ and note that $\widehat{f_x}(\xi)=e^{-2\pi i x\xi}\widehat{f_0}(\xi)$ so that $f_x\in PW^p(\Omega)$. Moreover, since $\mu$ is sampling and $1\leqslant p<\infty$, we get
\begin{equation*}
\mu(B_R(x))\leqslant\mu((x-R,x+R)^d)\lesssim\norm{f_x}_{L^p(\mu)}^p\asymp\norm{f_x}_p^p=\norm{f_0}_p^p.
\end{equation*}

Next, to construct $f_0$, we take $\eps>0$ such that $(-\eps,\eps)^d\subset\Omega$. If we find $\ffi\in PW^p(-\eps,\eps)$ such that, for every $R$, $\abs{\ffi}\gtrsim\chi_{(-R,R)}$ then $f_0(x_1,\dots,x_d)=\ffi(x_1)\times\dots\times\ffi(x_d)$ will do.

Now let $\psi:=\chi_{(-\eps/2,\eps/2)}*\chi_{(-\eps/2,\eps/2)}$ so that 

--- $\psi$ has support $(-\eps,\eps)$;

--- $\psi=\widehat{\ffi_0}$ where $\displaystyle\ffi_0(t)=\left(\frac{\sin(\pi\eps x)}{\pi x}\right)^2$; hence, in particular, $\ffi_0\in PW^p(-\eps,\eps)$;

--- $\ffi_0$ is continuous, nonnegative, and $\ffi_0(x)=0$ if and only if $x=k/\eps$, $k\in\Z\setminus\{0\}$. 

Finally, we set $\ffi(x):=\ffi_0(x)+\ffi_0(x+1/2\eps)$.  Then $\ffi$ is a continuous function in $PW^p(-\eps,\eps)$ that never vanishes. This means that, for every $R>0$, $\ffi\gtrsim\chi_{(-R,R)}$. The proof is thus complete.
\end{proof}

We now show that for any sampling measure we can extract a weighted sampling set. The argument mirrors that of \cite{MR1676043} for functions on the Bargmann-Fock space.

\begin{theorem}\label{extract}
Let $\Omega\subset\R^d$ be a convex centered symmetric body, $1\leqslant p<\infty$ and $\mu$ a Borelian measure in $\R^d$.
Let $r>0$  and $\{Q^r_n\}_n$ a (measurable) covering of $\R^d$ with finite multiplicity and $\sup_n\diam Q^r_n\leqslant 
r$, and set $\mu_r^*:=\sum_n\mu(Q^r_n)\dirac_{a^r_n}$ where $a^r_n\in Q^r_n$ are arbitrary points.

Then there exists a constant $C_{\Omega,p}>0$ such that $\mu$ is sampling for $PW^p(\Omega)$ if and only if $\mu_r^*$ is sampling for $PW^p(\Omega)$ when $r<C_{\Omega,p}$.
\end{theorem}

\begin{remark}\label{rmkballs}
Note that, in any direction of the Theorem, $\sup_{x\in\R^d}\mu(B_R(x))<\infty$ holds for all $R>0$. Indeed, if $\mu$ is sampling for $PW^p(\Omega)$, this is Lemma \ref{supr}. On the other hand, if $\mu_r^*$ is sampling for $PW^p(\Omega)$, Lemma \ref{supr} applied to $\mu_r^*$ reads $\sup_{x\in\R^d}\mu^*_r(B_R(x))<\infty$ for all $R>0$, and thus
\begin{equation*}
\mu(B_\rho(x))\leqslant\sum_{Q^r_n\cap B_\rho(x)\neq\emptyset}\mu(Q^r_n)\leqslant\mu^*_r(B_{\rho+r}(x))\leqslant C_{\rho,r},\quad\forall\rho>0,x\in\R^d,
\end{equation*}
since $\{Q^r_n\}_n$ is a covering with $\sup_n\diam Q^r_n\leqslant r$.	
\end{remark}

\begin{proof}[Proof of Theorem \ref{extract}]
Let $f\in PW^p(\Omega)$. First note that since $\{Q^r_n\}_n$ is a covering of $\R^d$ of finite multiplicity then
\begin{equation}
\label{partition}\int_{\R^d}\abs{f(x)}^p\d\mu(x)\asymp\sum_n\int_{Q^r_n}\abs{f(x)}^p\d\mu(x).
\end{equation}
It follows that, $\mu$ is sampling if 
\begin{equation}
\label{eq:musampling}\int_{\R^d}\abs{f(x)}^p\d x\asymp\sum_n\int_{Q^r_n}\abs{f(x)}^p\d\mu(x),
\end{equation}
and by definition $\mu_r^*$ is sampling if 
\begin{equation}
\label{eq:mursampling}\int_{\R^d}\abs{f(x)}^p\d x\asymp\sum_n\abs{f(a_n^r)}^p \mu(Q_n^r).
\end{equation}

Next, using the inequality $(a+b)^p\leqslant2^{p-1}(a^p+b^p)$ in each $Q^r_n$, we see that
\begin{multline*}
\sum_n\int_{Q^r_n}\abs{f(x)}^p\d\mu(x)\leqslant2^{p-1}\sum_n\int_{Q^r_n}\abs{f(x)-f(a^r_n)}^p+\abs{f(a^r_n)}^p\d\mu(x)\\
\lesssim I(r)+\sum_n\abs{f(a_n^r)}^p\mu(Q_n^r)
\end{multline*}
where $I(r):=\sum_n\int_{Q^r_n}\abs{f(x)-f(a^r_n)}^p\d\mu(x)$. Similarly, 
\begin{multline*}
\sum_n\abs{f(a^r_n)}^p\mu(Q^r_n)=\sum_n\int_{Q^r_n}\abs{f(a^r_n)}^p\d\mu(x)\\
\leqslant 2^{p-1}\sum_n\int_{Q^r_n}\abs{f(x)-f(a^r_n)}^p+\abs{f(x)}^p\d\mu(x)
\lesssim I(r)+\sum_n\int_{Q^r_n}\abs{f(x)}^p\d\mu(x).
\end{multline*}

Hence, to prove that \eqref{eq:musampling} and \eqref{eq:mursampling} are equivalent for some small value of $r$ it is enough to show that $I(r)$ is also sufficiently small. Precisely, we will now show that there exists a function $\phi(r)$ with $\phi(r)\underset{r\to0}{\to}0$ such that
\begin{equation}
\label{wrtr}I(r)\leqslant\phi(r)\int_{\R^d}\abs{f(x)}^p\d x.
\end{equation}
We proceed to prove \eqref{wrtr}. Since $f\in PW^p(\Omega)$ then, by Proposition \ref{PW}, it has a complex extension $f(x+iy)$ which is an entire function of exponential type (with constants depending on $\Omega$). Hence, in particular, $f$ is harmonic on $\C^d$ and satisfies the mean value theorem that we write in the form $f=f\ast\frac{1}{\abs{B_1(0)}}\chi_{B_1(0)}$. Further, iterating this formula, we get $f=f\ast\varphi$ with $\varphi=\frac{1}{\abs{B_1(0)}}\chi_{B_1(0)}\ast\dots\ast\frac{1}{\abs{B_1(0)}}\chi_{B_1(0)}$. Finally, notice that if we iterate
sufficiently many times, $\varphi$ is a compactly supported function of class $\Cc^1$. Thus, for each $n$ and $x \in Q^r_n$ we may write
\begin{align*}
\abs{f(x)-f(a^r_n)}^p&=\abs{(f\ast\varphi)(x)-(f\ast\varphi)(a^r_n)}^p
\\
&\leqslant\left(\int_{\R^d\times\R^d}\abs{f(y+iz)(\varphi(x-y-iz)-\varphi(a^r_n-y-iz)}\d y\d z\right)^p.
\end{align*}
Let $r_0>0$ be such that $\supp(\varphi)\subset B_{r_0}(0)$. Since $\abs{x-a^r_n}\leqslant\diam Q^r_n\leqslant r$ and both $x$ and $a^r_n$ are real, then the function $\varphi(x-y-iz)-\varphi(a^r_n-y-iz)$ is supported in $A_r(x):=\{y+iz\in\C^d:\abs{y-x}<r+r_0,\abs{z}<r_0\}$. This gives us
\begin{equation*}
\abs{f(x)-f(a^r_n)}^p\leqslant\left(\int_{A_r(x)}\abs{f(y+iz)(\varphi(x-y-iz)-\varphi(a^r_n-y-iz)}\d y\d z\right)^p.
\end{equation*}
Next, we apply Hölder's inequality
\begin{equation*}
\abs{f(x)-f(a^r_n)}^p\leqslant\norm{\varphi(x-\cdot-i\cdot)-\varphi(a^r_n-\cdot-i\cdot)}^p_{L^{p'}(A_r(x))}\int_{A_r(x)}\abs{f(y+iz)}^p\d y\d z,
\end{equation*}
where $1/p'+1/p=1$. We bound $\norm{\varphi(x-\cdot-i\cdot)-\varphi(a^r_n-\cdot-i\cdot)}^p_{L^{p'}(A_r(x))}$ by using the fact that $\varphi\in\Cc^1_c$ and then applying the Mean Value Theorem 
\begin{equation*}
\abs{\varphi(x-y-iz)-\varphi(a^r_n-y-iz)}\leqslant r\norm{\nabla\varphi}_\infty,
\end{equation*}
which holds for all $y+iz\in\C^d$. Therefore 
\begin{equation*}
\norm{\varphi(x-\cdot,\cdot)-\varphi(a^r_n-\cdot,\cdot)}^p_{L^{p'}(A_r(x))}\lesssim r^p \abs{A_r(x)}^{p/p'}\asymp r^p(r+r_0)^{dp/p'}r_0^{dp/p'}=:\phi(r).
\end{equation*}
Running the sum over all $n$ we get
\begin{multline*}
I(r)=\sum_n\int_{Q^r_n}\abs{f(x)-f(a^r_n)}^p\d\mu(x)
\\
\leqslant\phi(r)\sum_n\int_{Q^r_n}\int_{A_r(x)}\abs{f(y+iz)}^p\d y\d z\d\mu(x)
\\
\lesssim\phi(r)\int_{\R^d}\int_{A_r(x)}\abs{f(y+iz)}^p\d y\d z\d\mu(x)\\
= \phi(r)\int_{\R^d}\int_{\abs{z}<r_0}\int_{\abs{x-y}<r+r_0}\abs{f(y+iz)}^p\d y\d z\d\mu(x).
\end{multline*}
since $\{Q^r_n\}_n$ has finite multiplicity. Tonelli's theorem then implies
\begin{multline*}
I(r)\lesssim\phi(r)\int_{\R^d}\int_{\abs{z}<r_0}\int_{\abs{x-y}<r+r_0}\d\mu(x)\abs{f(y+iz)}^p\d z\d y
\\
\lesssim\phi(r)\int_{\R^d}\int_{\abs{z}<r_0}\abs{f(y+iz)}^p\d z\d y
\end{multline*}
since $\sup_{x\in\R^d}\mu(B_{r+r_0}(x))<\infty$ (cf. Remark \ref{rmkballs}). Finally applying Proposition \ref{PPL} we get
\begin{equation*}
I(r)\lesssim\phi(r)\int_{\R^d}\abs{f(x)}^p\d x.
\end{equation*}
Eventually multiplying $\phi(r)$ by constants, this gives \eqref{wrtr}.
\end{proof}

\begin{remark}
Note that the only impediment to extend Theorem \ref{extract} to the case $p=\infty$ is Remark \ref{rmkballs} (which does not apply for $p=\infty)$. However, if we suppose in addition that $\sup_{x\in\R^d}\mu(B_R(x))<\infty$ for some $R>0$ (which is a weaker condition to the conclusion of Remark \ref{rmkballs}) then we get back the result of Theorem \ref{extract}.
\end{remark}

\section{Sampling trajectories and spiraling curves}

\subsection{Sampling trajectories}\label{sec_traj}
A {\it curve} is a measurable map $\curve:\R\to\R^d$ (which we do not require to be continuous). A {\it trajectory} $\traj$ is the image of a curve: $\traj:=\curve(\R)$. The restriction of the one-dimensional Hausdorff measure to $\traj$ is denoted $\mu_\traj:=\Hh^1_\traj$. A trajectory $\traj$ is called {\it regular} if there exists $r_0>0$ such that for all $r\leqslant r_0$ there is a constant $c_r>0$ for which
\begin{equation}
\label{regradinf}\inf_{x\in\traj}\mu_\traj(B_r(x))\geqslant c_r,
\end{equation}
and
\begin{equation}
\label{regradsup}\sup_{x\in\R^d}\mu_\traj(B_{r_0}(x))<+\infty.
\end{equation}

A trajectory $\traj$ is called a {\it sampling trajectory} for $PW^p(\Omega)$ if $\mu_\traj$ is a sampling measure, i.e.,
\begin{equation*}
\norm{f}^p_p\asymp\int_{\traj}\abs{f(x)}^p\d\Hh^1(x),\quad f\in PW^p(\Omega),
\end{equation*}
with the usual modification for $p=\infty$.

\subsection{Characterization of sampling trajectories}
Using the general result we proved for sampling measures, Theorem \ref{extract}, we can show how to extract a sampling set from a sampling trajectory. This is Theorem D presented in the Introduction and that we recall here for the convenience of the reader:
\begin{thmcontdiscrete}{\em
Let $\traj\subset\R^d$ be a regular trajectory, $\Omega\subset\R^d$ bounded with positive measure. Then $\traj$ is a sampling trajectory for $PW^2(\Omega)$ if and only if there exists $\aset\subset\traj$ that is a sampling set for $PW^2(\Omega)$. Moreover, $\aset$ can always be chosen to be separated.}
\end{thmcontdiscrete}
\begin{remark}
Although the statement concerns $PW^2$, we remark that the following proof is still valid for any $PW^p$ with $1\leqslant p\leqslant\infty$.
\end{remark}
\begin{proof}
Due to Theorem \ref{extract} it is enough to show that there exists a separated set $\aset\subset\traj$ such that its point measure $\delta_\aset=\sum_{\lambda\in\aset}\delta_\lambda$ is equivalent to 
$\mu^*_r=\sum_n\mu_\traj(Q^r_n)\dirac_{a^r_n}$ where $a^r_n\in Q^r_n$ and $\{Q^r_n\}_n$ is a covering of 
finite multiplicity of $\R^d$ with $\sup_n\diam Q^r_n<r$, and $r$ is small enough. Moreover, in this case we 
can discard those $Q_n$ which do not meet $\traj$ (since $\mu_\traj(Q_n)=0$ when $Q_n\cap\traj=\emptyset$), and just consider coverings of $\traj$.

Given $r>0$, let $\{a^r_n\}_n\subset\traj$ be maximal with respect to $\abs{a^r_n-a^r_m}\geqslant r,n\neq m$. Hence, the family $\{B_r(a^r_n)\}_n$ is a covering of $\traj$ with finite multiplicity (the covering number being bounded by $4^d$). Next, set $\mu_r^*:=\sum_n\mu_\traj(B_r(a^r_n))\dirac_{a^r_n}$ and $\aset_r:=\{a^r_n\}_n$. Taking $r<r_0$ \eqref{regradinf} reads $\mu_\traj(B_r(a^r_n))\geqslant c_r$ for all $n$, and therefore $\mu_r^*\gtrsim\sum_n\dirac_{a^r_n}=\delta_{\aset_r}$. On the other hand, when $r<r_0$ the condition \eqref{regradsup} reads $\sup_n\mu(B_r(a^r_n))<\infty$ so that also $\mu_r^*\lesssim\delta_{\aset_r}$. In sum for all $r<r_0$ we can construct a finite multiplicity $r$-covering $\{Q^r_n\}_n$ together with a separated set $\aset_r$ such that $\mu^*_r=\sum_{\lambda\in\aset_r}\mu_\traj(Q^r_n)\delta_{\lambda}\asymp\sum_{\lambda\in\aset_r}\delta_{\lambda}=\delta_{\aset_r}$.
\end{proof}

\begin{remark}
\label{rem_suf_dense}
The proof of Theorem D shows that if $\traj$ is a sampling trajectory, then
\emph{any} separated and sufficiently dense $\aset\subset\traj$ is sampling for $PW^2(\Omega)$.
\end{remark}

Theorem D shows the equivalence of two possible models for the mobile sampling problem \cite{gruv15,uv2,uv1}. As a consequence, sampling trajectories lead to concrete reconstruction strategies and numerical implementations, where a finite set of samples is used, and reconstruction is achieved within a precise numerical accuracy \cite{MR3254605,adgaro16,MR3209727,MR2231856, MR2286012, MR2679810}.

\subsection{Spiraling curves}\label{sec_spiraling}
Let us now describe precisely what we mean by a {\it spiraling} trajectory. The properties we list below have been chosen so that several natural sampling strategies are covered. Concentric circles and the Archimedes spiral will be shown to satisfy these properties (see Proposition \ref{prop_ex}). Other examples are depicted in Figure \ref{fig:parallel}.

A regular trajectory $\traj$ is called spiraling if the following conditions hold:
\begin{enumerate}
\renewcommand{\theenumi}{\roman{enumi}}
\item ({\it Escape cone}). There exist $\alpha\in(0,1/4)$ and $\beta\in(0,1)$ such that the portion of $\traj$ contained in the cone
\begin{equation}
\label{eq_cone}S_{\alpha,\beta}:=\{(\rho\cos2\pi\theta,\rho\sin2\pi\theta):\rho\geqslant0\text{, }\beta-\alpha\leqslant\theta\leqslant\beta+\alpha\}
\end{equation}
can be parametrized in polar coordinates as
\begin{equation}
\label{eq_curve}\curve(\theta)=(\rho(\theta)\cos2\pi\theta,\rho(\theta)\sin2\pi\theta)
\end{equation}
with $\theta\in\bigcup_{k\in\N}[k+\beta-\alpha,k+\beta+\alpha]$ and where $\rho(\theta)$ is a non negative 
function of class $C^2$ in each interval. In particular, this means that, inside the escape cone, each {\it piece of the trajectory} $\{\curve(\theta):\theta\in[k+\beta-\alpha,k+\beta+\alpha]\}$, joins one of the boundary lines of the cone to the other one.

We denote by $\vec{l}=(\cos2\pi\beta,\sin2\pi\beta)$ the bisector vector of the escape cone. Every piece of the trajectory in the escape cone intersects the half-line $\R_+\vec{l}$ once.

\item ({\it Asymptotic radial monotonicity}). There exists $k_\rho$ such that for any $\theta\in[\beta-\alpha,\beta+\alpha]$ the sequence $\rho(\theta+k)$ is strictly increasing for $k\geqslant k_\rho$.

In particular, inside the escape cone, two different pieces of the trajectory, $\{\curve(\theta):\theta\in[k+\beta-\alpha,k+\beta+\alpha]\}$ and $\{\curve(\theta):\theta\in[l+\beta-\alpha,l+\beta+\alpha]\}$ with $l\neq k$ and $l,k\geqslant k_\rho$, do not intersect.

\item ({\it Asymptotic flatness}). The curvature of $\curve(\theta)$, denoted by $\kappa(\theta)$, tends to $0$ as $\theta\to+\infty$, i.e. for all $\eps>0$ there exists $k_\eps\in\N$ such that 
$\kappa(\theta)<\eps$ whenever $\theta\in[k+\beta-\alpha,k+\beta+\alpha]$ with $k\geqslant k_\eps$.

\item ({\it Asymptotic equispacing}). There are two parameters $\sep,\rho_0>0$ such that the sequence $\sep_k=\rho(k+\beta)$ has the property 
\begin{equation*}
\lim_k\sep_k-\sep k=\rho_0.
\end{equation*}

\item ({\it Asymptotic velocity}). There exists a direction $\vec{d}\in\mathbb{S}^1$ non-collinear with $\vec{l}$ such that
\begin{equation*}
\lim_{k\to+\infty}\frac{\curve'(k+\beta)}{\abs{\curve'(k+\beta)}}=\vec{d}.
\end{equation*}
\end{enumerate}

\begin{figure}[ht]
\begin{tikzpicture}[scale=0.5, every node/.style={scale=0.5}]
\draw[->](-3,0)--(10,0);
\draw[->](0,-3)--(0,10);
\draw[rotate=30](7.3,-7.3)--(0,0)--(7.3,7.3)node[right]{$S_{\alpha,\beta}$};
\draw[green,dashed,rotate=30](0,0)--(10,0);
\node at (2.7,1.2) {$\sep_2$};
\node at (5,2.5) {$\sep_3$};
\draw[fill=purple,fill opacity=0.1](0,0)--(0.7,0) arc(0:30:0.7);
\node at (1,0.2) {$\beta$};
\draw[fill=purple,fill opacity=0.1,rotate=30](0,0)--(0.8,0) arc(0:45:0.8);
\node at (0.6,0.8) {$\alpha$};
\draw[purple, rotate=30](1,1)..controls(2,0)and(0.5,0)..(3,-3)
(3,3)..controls(3.5,0)and(2.75,0)..(5,-5)
(5,5)..controls(5,0)and(7,0)..(7,-7);
\draw[blue,->,rotate=30](5.9,0)--(5.3,2.5);
\node at (4.5,4) {$\vec{d}$};
\end{tikzpicture}
\caption{Sketch of an spiraling curve.}
\end{figure}
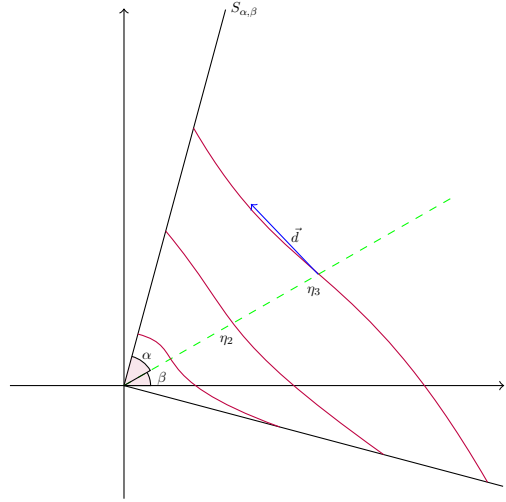

\begin{figure}[ht]
\begin{tikzpicture}[scale=0.5, every node/.style={scale=0.5}]
\draw[->](-3,0)--(10,0);
\draw[->](0,-1)--(0,7);
\draw[rotate=30](9,-3)--(0,0)--(9,3);
\draw[-,purple](1,0)--(1,3);
\draw[-,purple](2,3) arc(0:180:0.5);
\draw[-,purple](2,0)--(2,3);
\draw[-,purple](2,0) arc(180:360:0.5);
\draw[-,purple](3,0)--(3,5);
\draw[-,purple](4,5) arc(0:180:0.5);
\draw[-,purple](4,0.5)--(4,5);
\draw[-,purple](4,0.5) arc(180:360:0.5);
\draw[-,purple](5,0.5)--(5,6.5);
\draw[-,purple](6,6.5) arc(0:180:0.5);
\draw[-,purple](6,1)--(6,6.5);
\end{tikzpicture}
\hspace*{3em}
\begin{tikzpicture}[scale=0.5, every node/.style={scale=0.5}]
\draw[->](-5,0)--(5,0);
\draw[->](0,-5)--(0,5);
\draw[purple](-1,-1)--(-1,1)--(1,1)--(1,-1)--(-1,-1);
\draw[purple](-2,-2)--(-2,2)--(2,2)--(2,-2)--(-2,-2);
\draw[purple](-3,-3)--(-3,3)--(3,3)--(3,-3)--(-3,-3);
\draw[purple](-4,-4)--(-4,4)--(4,4)--(4,-4)--(-4,-4);
\end{tikzpicture}
\caption{A spiraling curve based on a set of parallel lines in a sector (left) and a spiraling curve made of concentric squares (right).}
\label{fig:parallel}
\end{figure}
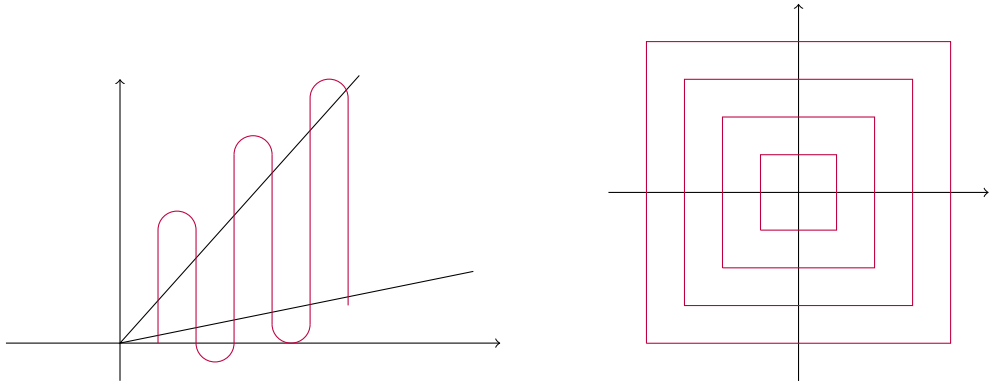

The number $\tau:=\sep\sqrt{1-\scal{\vec{l},\vec{d}}^2}$ is called the {\it asymptotic separation} of $\traj$. For short, we say that $\traj$ is a spiraling trajectory with asymptotic velocity $\vec{d}$ and asymptotic separation $\tau$. Note that those parameters may not be unique.

\begin{remark}
\label{rem_rot}
The class of spiraling curves is invariant under rotations. Indeed, if we rotate a spiraling curve by angle of $2\pi\theta_0$ with $\theta_0\in[0,1)$ then the resulting curve is spiraling with parametrization in the escape cone with $\tilde\beta=\beta-\theta_0$, $\tilde\alpha=\alpha$, and
\begin{equation*}
\tilde{\curve}(\theta)=(\tilde{\rho}(\theta)\cos2\pi\theta,\tilde{\rho}(\theta)\sin2\pi\theta),\quad\theta\in\bigcup_k[k+\tilde{\beta}-\alpha,k+\tilde{\beta}+\alpha],
\end{equation*}
where $\tilde{\rho}(\theta)=\rho(\theta+\theta_0)$. The rotated curve has asymptotic velocity $R_{2\pi \theta_0}\vec{d}$, while the parameters of asymptotic equispacing $\sep$, $\rho_0$ and the asymptotic separation $\tau$ remain unaltered.

Further, spiraling curves are also invariant under some reasonable smooth perturbations of the escape cone. More precisely, let us first assume that $\beta=0$ so that the escape cone is $S_{\alpha,0}$ and let $f:S_{\alpha,0}\to S_{\alpha',0}$ be a $\mathcal{C}^2$ one-to-one function. Assume the following:

\begin{enumerate}
\renewcommand{\theenumi}{\roman{enumi}}
\item $f((\rho\cos2\pi\alpha,\pm\rho\sin2\pi\alpha))=(\rho\cos2\pi\alpha',\pm\rho\sin2\pi\alpha')$, i.e. $f$ sends the boundary of $S_{\alpha,0}$ to the boundary of $S_{\alpha',0}$.

\item $f(x_1,0)=(\ffi(x_1),0)$ for
an asymptotically affine function
$\ffi:\R\to\R$, that is,
$\ffi(x_1)-(ax_1+b)\to0$ when $x_1\to+\infty$ for some constants $a>0$, $b\geqslant0$.

\item If $\rho:[-2\pi\alpha,2\pi\alpha]\to[0,+\infty)$ then the curve $f((\rho(\theta)\cos2\pi\theta,\rho(\theta)\sin\theta))$ admits a parametrization in polar coordinates $(\tilde{\rho}(\theta)\cos2\pi\theta,\tilde{\rho}(\theta)\sin2\pi\theta)$ with $\tilde{\rho}:[-2\pi\alpha',2\pi\alpha']\to[0,+\infty)$.

Note that, as $f$ is one-to-one, sends the bisector on itself and it behaves asymptotically like the increasing linear function $ax_1+b$ in $(x_1,0)$, $f$ asymptotically preserves radial monotonicity.

\item  The Jacobian of $f$ is uniformly bounded from above and below, i.e. there exist $A,B>0$ such that $A\abs{y}\leqslant\abs{J_xf(y)}\leqslant B\abs{y}$ for all $x\in S_{\alpha,0}$, $y\in\R^2$. Moreover $J_{(x_1,0)}f\to M$ when $x_1\to+\infty$, and $M\vec{d}\not=(1,0)$.

\item The Hessian of $f$ goes to $0$ when $x\to\infty$, i.e. $H_xf\to 0$ when $\abs{x}\to+\infty$.
\end{enumerate}

Then, if $\traj$ is a spiraling trajectory with asymptotic velocity $\vec{d}$ and asymptotic equispacing $\sep$, $f(\traj)$ is a spiraling trajectory with asymptotic equispacing $a\sep$ and asymptotic velocity $\frac{M\vec{d}}{\abs{M\vec{d}}}$. As a consequence, we have for example that, combining these smooth perturbations with rotations, spiraling curves are invariant by any linear invertible transformation in $\R^2$.
\end{remark}

\subsection{Examples of spiraling curves}
\begin{proposition}\label{prop_ex}
Let $\sep>0$. Then, the Archimedes spiral $A^\sep$ and the union of circles $O^\sep$ are spiraling 
trajectories with asymptotic separation $\sep$. Further, any $\vec{d}\in\mathbb{S}^1$ can be taken as the asymptotic velocity.
\end{proposition}
\begin{proof}
\noindent {\it Step 1 (Regularity).}
We show first that $A^\sep$ and $O^\sep$ are regular trajectories in the sense of $\S$ \ref{sec_traj}. We
define $\rho_1(\theta):=\sep\theta$ and $\rho_2(\theta):=\sum_{k\in\N}\sep k\chi_{[k,k+1)}(\theta)$ so that
\begin{gather*}
A^\sep=\{(\rho_1(\theta)\cos2\pi\theta,\rho_1(\theta)\sin2\pi\theta):\theta\geqslant0\},
\\
O^\sep=\{(\rho_2(\theta)\cos2\pi\theta,\rho_2(\theta)\sin2\pi\theta):\theta\geqslant0\}.
\end{gather*}

Let us begin by proving that \eqref{regradinf} holds for $A^\sep$. Take $r \in (0,1)$ and $x \in A^\sep$ and assume initially that $\abs{x}\leqslant r/2$. Then
\begin{align*}
\mu(B_r(x))&\geqslant\mu(B_{r/2}(0))=\int_{0}^{r/2}\sqrt{\rho_1'(\theta)^2+(2\pi\rho_1(\theta))^2}\d\theta
\\
&\geqslant\int_{0}^{r/2}2\pi\rho_1(\theta)\d\theta=\pi r^2/4.
\end{align*}
For $\abs{x}>r/2$, we proceed as follows. Let $y\in B_r(x)\cap A^\sep$ and write
\begin{gather*}
x=(\rho_1(\theta_0)\cos2\pi\theta_0,\rho_1(\theta_0)\sin2\pi\theta_0),\quad\theta_0\geqslant0,
\\
y=(\rho_1(\theta)\cos2\pi\theta,\rho_1(\theta)\sin2\pi\theta),\quad\theta\geqslant0.
\end{gather*}
Therefore,
\begin{align*}
\abs{x-y}&\leqslant\abs{\rho_1(\theta_0)\cos2\pi\theta_0-\rho_1(\theta)\cos2\pi\theta}+\abs{\rho_1(\theta_0)\sin2\pi\theta_0-\rho_1(\theta)\sin2\pi\theta}
\\
&\leqslant\abs{\rho_1(\theta_0)}(\abs{\cos2\pi\theta_0-\cos2\pi\theta}+\abs{\sin2\pi\theta_0-\sin2\pi\theta})
\\
&\quad+\abs{\rho_1(\theta_0)-\rho_1(\theta)}(\abs{\cos2\pi\theta}+\abs{\sin2\pi\theta})
\\
&\leqslant(\sep+4\pi\rho_1(\theta_0))\abs{\theta_0-\theta}
\end{align*}
by the Mean Value Theorem. Let us set $r_0:=\frac{r}{\sep+4\pi\rho_1(\theta_0)}$, so that $\abs{\theta_0-\theta}\leqslant r_0$ implies $\abs{x-y}<r$.
This allows us to bound the arc-length of $B_r(x)\cap
A^\sep$ from below by
\begin{align*}
\mu(B_r(x))&\geqslant\int_{\theta_0-r_0}^{\theta_0+r_0}\sqrt{\rho_1'(\theta)^2+(2\pi\rho_1(\theta))^2}
\d\theta
\\
&\geqslant\int_{\theta_0-r_0}^{\theta_0+r_0}2\pi\rho_1(\theta)\d\theta=4\sep\pi\theta_0r_0
\\
&\gtrsim\frac{\rho_1(\theta_0)}{1+\rho_1(\theta_0)}r=\frac{\abs{x}}{1+\abs{x}}r\geqslant\frac{r^2}{2+r}\geqslant\frac{r^2}{3}
\end{align*}
since  $\abs{x}>r/2$, and $r\in(0,1)$. In conclusion, for every $r<1$ and every $x$, $\mu(B_r(x))\gtrsim r^2$.

The argument for $O^\sep$ is similar, this time replacing $\rho_1(\theta)$ with $\rho_2(\theta)$.
We now show \eqref{regradsup}, beginning again with case of $A^\sep$. Note first that it is enough to bound $\mu_{A^\sep}(B_r(x))$ uniformly for $x\in A^\sep$, and some $r>0$. Indeed, since $\gap(A^\sep)=\sep/2$, for every $y\in\R^2$, there exists $x\in A^\sep$ such that $B_r(y)\subset B_{r+\sep/2}(x)$. In addition, any ball of radius $R>r$ can be covered with $(2R/r+2)^2$ balls of radius $r$ (by taking a covering $\{B_r(a_n)\}_n$ maximal with respect to $\abs{a_n-a_m}\geqslant r$ as in the proof of Theorem D).

Let $x=(\rho_1(\theta_0)\cos 2\pi \theta_0, \rho_1(\theta_0)\sin 2\pi \theta_0)$ and let 
$r:=\min\{\sep/100,1\}$. If $\abs{x}\leqslant100$ we simply bound
\begin{equation*}
\mu_{A^\sep}(B_r(x))\leqslant\mu_{A^\sep}(B_{101}(0)).
\end{equation*}
Let us assume then $\abs{x}\geqslant100$. We claim that
\begin{equation}
\label{eq_aaa}
B_r(x) \cap A^\sep\subset\{(\rho_1(\theta)\cos 2\pi \theta, \rho_1(\theta)\sin 2\pi 
\theta):\theta\in(\theta_0-1/\abs{x},\theta_0+1/\abs{x})\}.
\end{equation}
Indeed, if $y=(\rho_1(\theta)\cos2\pi\theta,\rho_1(\theta)\sin2\pi\theta)\in B_r(x)\cap A^\sep$, 
\begin{equation*}
\sep\abs{\theta_0-\theta}=\abs{\rho_1(x)-\rho_1(y)}=\abs{\abs{y}-\abs{x}}\leqslant\abs{x-y}<r\leqslant\eta/100,
\end{equation*}
and therefore $\abs{\theta_0-\theta}\leqslant1/100$. Secondly, since $\abs{x-y}<r$, a clockwise rotation by the angle of $y$ gives
\begin{equation*}
\abs{x-y}=\sqrt{\abs{\rho_1(\theta_0)\cos2\pi(\theta_0-\theta)-\rho_1(\theta)}^2+\abs{\rho_1(\theta_0)\sin2\pi(\theta_0-\theta)}^2}<r,
\end{equation*}
and in particular $\abs{\rho_1(\theta_0)\sin2\pi(\theta_0-\theta)}=\abs{x}\abs{\sin2\pi(\theta_0-\theta)}<r\leqslant1$. Using the bound $\abs{\sin2\pi(\theta-\theta_0)}\geqslant\abs{\theta-\theta_0}$ - valid for
$\abs{\theta-\theta_0}\leqslant1/100$, we conclude that
\begin{equation*}
\abs{\theta_0-\theta}\leqslant\abs{\sin2\pi(\theta-\theta_0)}\leqslant\abs{x}^{-1},
\end{equation*}
and \eqref{eq_aaa} follows. Now we can estimate,
\begin{multline*}
\mu_{A^\sep}(B_r(x))\leqslant\int_{\theta_0-1/\abs{x}}^{\theta_0+1/\abs{x}}\sqrt{\rho_1'(\theta)^2+(2\pi\rho_1(\theta))^2}\d\theta
\\
\leqslant\int_{\theta_0-1/\abs{x}}^{\theta_0+1/\abs{x}}\rho_1'(\theta)+2\pi\rho_1(\theta)\d\theta=\int_{\theta_0-1/\pi\abs{x}}^{\theta_0+1/\pi\abs{x}}\sep+2\pi\sep\theta\d\theta
\\
\leqslant\frac{2\sep}{\abs{x}}+\frac{4\pi\sep\theta_0}{\abs{x}}\lesssim 1.
\end{multline*}
Since $\abs{x}\geqslant100$ and $\sep\theta_0=\abs{x}$. The proof for $O^\sep$ follows similarly.

\noindent {\it Step 2 (The other conditions).}
We choose $\vec{d}\in\mathbb{S}^1$ and verify that $A^\sep$ and $O^\sep$ satisfy the conditions of $\S$\ref{sec_spiraling} with asymptotic velocity $\vec{d}$. Write $\vec{d}=(-\sin2\pi\theta_0,\cos2\pi\theta_0)$ with $\theta_0\in[0,1)$, and $\beta=2\pi \theta_0$. Thus $\vec{l}=\vec{d}_\perp=(\cos2\pi\theta_0,\sin2\pi\theta_0)$.

We start with $A^\sep$. For the escape cone \eqref{eq_cone}, we choose any $\alpha\in(0,1/4)$, and parametrize the portion of the curve inside the cone with $\curve=\curve_{A^\sep}$ as in \eqref{eq_curve}, and using the function $\rho_1(\theta)=\sep\theta$ restricted to each $[k+\theta_0-\alpha,k+\theta_0+\alpha]$, $k\in\N$. Then $\rho_1(\theta)$ is strictly increasing and in particular the monotonicity condition is satisfied. Also, the curvature
\begin{gather*}
\kappa_{A^\sep}(\theta)=\frac{2+(2\pi\theta)^2}{\sep(1+(2\pi\theta)^2)^{3/2}}
\end{gather*}
converges to $0$ as $\theta\to\infty$. Hence, the asymptotic flatness condition holds.
For the asymptotic equispacing condition, we let $\rho_0:=\sep\theta_0$ and simply note that
\begin{equation*}
\rho_1(k+\beta)=\rho_1(k+\theta_0)=\sep(k+\theta_0)=\sep k+\rho_0,
\end{equation*}
for all $k\in\N$, so there is nothing to prove. Finally, we check that $\vec{d}$ is the asymptotic velocity:
\begin{align*}
\lim_k&\frac{\curve_{A^\sep}'(k+\beta)}{\abs{\curve_{A^\sep}'(k+\beta)}}=\lim_k\frac{\curve_{A^\sep}'(k+\theta_0)}{\abs{\curve_{A^\sep}'(k+\theta_0)}}
\\
&=\lim_k\frac{\rho_1'(k+\theta_0)(\cos2\pi\theta_0,\sin2\pi\theta_0)+2\pi\rho_1(k+\theta_0)(-\sin2\pi\theta_0,\cos2\pi\theta_0)}{\sqrt{\rho_1'(k+\theta_0)^2+(2\pi\rho_1(k+\theta_0))^2}}
\\
&=\lim_k\frac{(\cos2\pi\theta_0,\sin2\pi\theta_0)+2\pi(k+\theta_0)(-\sin2\pi\theta_0,\cos2\pi\theta_0)}{\sqrt{1+2\pi(k+\theta_0)^2}}
\\
&=(-\sin2\pi\theta_0,\cos2\pi\theta_0)=\vec{d}.
\end{align*}
We now consider $O^\sep$. Since this curve is rotation invariant, we may assume that $\vec{d}=(0,1)$ and $\vec{l}=(1,0)$. The escape cone \eqref{eq_cone} is then parametrized with $\beta=0$ and any $\alpha\in(0,1/4)$ by setting $\rho_2(\theta)=\sum_k\sep k\chi_{[k-\alpha,k+\alpha]}(\theta)\chi_{[k,k+1)}(\theta)$. Then $\rho_2(\theta+k)$ is increasing for any $\theta\in[-\alpha,\alpha]$ and $\kappa_{O^\sep}(\theta)=\sum_k\frac{1}{\sep k}\chi_{[k-\alpha,k+\alpha]}(\theta)\chi_{[k,k+1)}(\theta)$. Therefore the monotonicity and asymptotic flatness conditions hold. The asymptotic equispacing condition also holds with $\rho_0=0$ because $\rho_2(k)=\sep k$ for all $k\in\Z$. Lastly, we calculate the asymptotic velocity:
\begin{equation*}
\lim_k\frac{\curve_{O^\sep}'(k)}{\abs{\curve_{O^\sep}'(k)}}=\lim_k\frac{\rho_2'(k)(1,0)+2\pi\rho_2(k)(0,1)}{\sqrt{\rho_2'(k)^2+(2\pi\rho_2(k))^2}}=\lim_k\frac{2\pi k(0,1)}{2\pi k}=(0,1)=\vec{d}.
\end{equation*}
\end{proof}

\subsection{Bessel bounds}
The following Bessel bounds follow from \cite[Theorem 3.2]{MR3528398}.

\begin{proposition}\label{JM}
Let $\traj$ be either the spiral $A^\sep$ or the concentric circles $O^\sep$. Then, 
there exists an independent constant $C>0$ such that for all $1\leqslant p\leqslant\infty$, $R>0$ and $f\in PW^p(\bar{B}_R(0))$
\begin{equation*}
\norm{f}_{L^p(\mu_\traj)}\leqslant C(\sep^{-1/p}+R^{1/p})\norm{f}_p
\end{equation*}
with the usual modifications when $p=\infty$. Moreover, the same conclusion holds for the rotated Archimedes spirals $\Rr_{2\pi\theta_0}A^\sep$.
\end{proposition}

\section{Necessary conditions for sampling on spiraling curves}
\subsection{General results}
In this section we derive necessary conditions for sampling on spiraling curves, comparing the asymptotic equispacing of such a curve to the diameter of the Fourier spectrum. The following key lemma shows that spiraling curves can be locally approximated by unions 
of lines.
\begin{lemma}\label{lemma_wl}
Let $\traj$ be a spiraling trajectory with asymptotic velocity $\vec{d}$ and asymptotic separation
$\tau$. Then the collection of equispaced parallel lines
\begin{equation}
\label{eq_equi_pl}
L_{\vec{d},\tau}:=\{t\vec{d}+\tau k\vec{d}_\perp:t\in\R,k\in\Z\}
\end{equation}
is a weak limit of translates of $\traj$, i.e. $L_{\vec{d},\tau}\in W(\traj)$.
\end{lemma}

\begin{proof}
Throughout the proof we use the notation of $\S$\ref{sec_spiraling}. 

\medskip

\noindent{\it Step 1.} {\it Reduction of the problem.}

\medskip

Let us first slightly simplify the setting. Since the notion of spiraling trajectory is invariant under rotation, we may assume that the escape cone is $S_{\alpha,0}$, thus, $\beta=0$ and $\vec{l}=\vec{i}:=(1,0)$. We then write $\vec{d}=(d_1,d_2)$. As $1=\abs{\vec{d}}^2=d_1^2+d_2^2$ and $\vec{d}$ is not collinear with $\vec{l}$, then $d_2\neq0$. Thus $\vec{i}=d_1\vec{d}-d_2\vec{d}_\perp$ and $\vec{d}_\perp=\frac{d_1}{d_2}\vec{d}-\frac{1}{d_2}\vec{i}$. It follows that, if $\sep=\tau/\abs{d_2}$, then
\begin{equation*}
L_{\vec{d},\tau}=\tilde L_{\vec{d},\sep}:=\{s\vec{d}+\sep j\vec{i}:s\in\R\text{, }j\in\Z\}.
\end{equation*}
Note that $\tau=\sep\abs{d_2}=\sep\sqrt{1-d_1^2}=\sep\sqrt{1-\scal{\vec{d},\vec{l}}^2}$.

\medskip

\noindent{\it Step 2.} {\it Reparametrization of the trajectory.}

\medskip

We set $\aset_k:=\traj-(\sep_k,0)$ and, from now on, we fix $R,\eps>0$. Our aim is to show that there exists $k_0$ (depending on $R,\eps$) such that; for all $k\geqslant k_0$,
\begin{gather}
\label{curve-lines}\aset_k\cap(-R,R)^2\subset \tilde L_{\vec{d},\sep}+B_\eps(0),
\\
\label{lines-curvesa}\tilde L_{\vec{d},\sep}\cap(-R,R)^2\subset\aset_k+B_\eps(0).
\end{gather}

\smallskip

First, a simple computation shows that
\begin{equation*}
\tilde L_{\vec{d},\sep}\cap(-R,R)^2\subset L_0:=\{t\vec{d}+\sep j\vec{i}:\ \abs{t}\leqslant T_R\text{, }j=-J_R,\ldots,J_R\}
\end{equation*}
with $T_R=\frac{R}{\abs{d_2}}$ and $J_R=\left\lceil\frac{2R}{\sep}\left(1+\abs{\frac{d_1}{d_2}}\right)\right\rceil$.
Instead of \eqref{lines-curvesa}, we will thus prove
\begin{equation}
\label{lines-curves}L_0\subset\aset_k+B_\eps(0).
\end{equation}

Next, as $\sep_k\to+\infty$ and $L_0$ is a bounded set, there is a $k_1$ such that, if $k\geqslant k_1$, the translate of the escape cone by $\sep_k$ contains $L_0$.  Moreover, as $\sep_k-k\sep\to\rho_0$, we can assume that, for $k\geqslant k_1$, $\abs{\sep_k-k\sep-\rho_0}<\eps/2$.
From now on, we will assume that $k\geqslant k_1$.

For each $n$ let $\psi_n:I_n\to\R^2$ be a re-parametrization by arc-length of $\curve(\theta)$ on the interval $[n-\alpha,n+\alpha]$ such that $0\in I_n$ and $\psi_n(0)=\curve(n)=(\sep_n,0)$. 
Write $I_n^\pm=I_n\cap\R^\pm$. As the restriction of $\psi_n$ to $I_n^-$ joins the line $\R(\cos\alpha,-\sin\alpha)$ to the point $(\sep_n,0)$ and $\sep_n\to+\infty$, it follows that the length of $I_n^-$, $\abs{I_n^-}\to+\infty$. Similarly, $\abs{I_n^+}\to+\infty$. (See Figure \ref{deltak}).

\begin{figure}[ht]
\begin{tikzpicture}[scale=0.5, every node/.style={scale=0.5}]
\draw[->](-3,0)--(10,0)
node[pos=0.8,below]{$\sep_n\to+\infty$};
\draw[->](0,-5)--(0,6);
\draw(5.5,-5.5)--(0,0)--(6.2,6.2);
\draw[purple](6,6)..controls(8,0)and(5,0)..(5,-5);
\draw[dashed,green](3.1,3.1)--(6.2,0)--(3.1,-3.1);
\draw[green](3,3)--(3.1,2.9)--(3.2,3)
(3,-3)--(3.1,-2.9)--(3.2,-3);
\end{tikzpicture}
\caption{The arc-length of $\psi_n$ is at least as large as twice the distance of $\sep_n$ to $S_{\alpha,0}$.}
\label{deltak}
\end{figure}
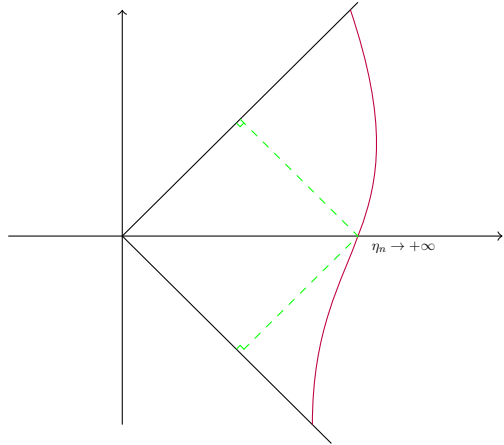

Therefore, there exists $n_1$ such that, for all $n\geqslant n_1$,
\begin{equation*}
[-2T_R,2T_R]\subset I_n.
\end{equation*}

A Taylor expansion of $\psi_n$ at $0$ reads
\begin{equation*}
\abs{\psi_n(t)-(\sep_n,0)-t\psi_n'(0)}\leqslant \frac{\abs{t}^2}{2}\sup_{s\in I_n}\abs{\psi''_n(s)}.
\end{equation*}
This implies
\begin{equation}
\label{eq:tayl}
\abs{\psi_n(t)-(\sep_n,0)-t\vec{d}}\leqslant\abs{t}\cdot\abs{\psi_n'(0)-\vec{d}\:}+\frac{\abs{t}^2}{2}\sup_{s\in I_n}\abs{\psi''_n(s)}.
\end{equation}
On the other hand, by definition we know that
\begin{equation*}
\psi_n'(0)=\frac{\curve'(n)}{\abs{\curve'(n)}},\quad\sup_{s\in I_n}\abs{\psi''_n(s)}=\sup_{s\in[n-\alpha,n+\alpha]}\kappa(s).
\end{equation*}
As $\curve$ is a spiraling curve, $\lim_n\abs{\psi_n'(0)-\vec{d}\:}=\lim_k\sup_{s\in I_n}\abs{\psi''_n(s)}=0$ and $\lim_n \sep_n-\sep n=\rho_0$. Therefore, there exists $n_2\geqslant n_1$ such that, for $n\geqslant n_2$
\begin{equation*}
\abs{\sep_n-\sep n-\rho_0}<\eps/6,\quad\abs{\psi_n'(0)-\vec{d}}<\eps/(12T_R),\quad\sup_{s\in I_n}\abs{\psi''_n(s)}<\eps/(12T_R^2).
\end{equation*}
It then follows from \eqref{eq:tayl} that, for all $t\in[-2T_R,2T_R]$, for $n\geqslant n_2$
\begin{equation}
\label{eq:eps}\abs{\psi_n(t)-(\sep n+\rho_0,0)-t\vec{d}}<\eps/2.
\end{equation}
In other words, the trajectory stays at distance at most $\eps/2$ from the segment joining the lines $\{x_2=-R\}$ to the line $\{x_2=R\}$, passing through the point $(\sep n+\rho_0,0)$ and directed by $\vec{d}$. See Figure \ref{fig:eps}.

\begin{figure}[ht]
\begin{tikzpicture}[scale=0.5, every node/.style={scale=0.5}]
\draw[blue](3,4)--(2,-4);
\draw[->,blue](0,0)--(0.15,1.5);
\node at (0.6,1.6) {$\vec{d}$};
\draw[red](3.5,3.9)..controls(2,2)and(3,-2)..(1.5,-4.1);
\draw[dashed,green](0,-4)--(4,-4);
\draw[dashed,green](0,4)--(4,4);
\draw[dashed,green](0,2)--(4,2);
\draw[dashed,green](0,-2)--(4,-2);
\draw[dashed,green](0,0)--(4,0);
\node at (-0.5,2) {$R$};
\node at (-0.5,4) {$2R$};
\node at (-0.5,0) {$0$};
\node at (-0.5,-2) {$-R$};
\node at (-0.5,-4) {$-2R$};
\draw[gray, opacity=0.2,fill=gray,fill opacity=0.2](2.25,4)--(1.25,-4) arc(180:360:0.75) (2.75,-4)--(3.75,4) arc(0:180:0.75);
\end{tikzpicture}
\caption{Illustration of \eqref{eq:eps}.}
\label{fig:eps}
\end{figure}
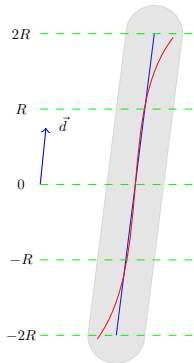

\medskip

\noindent{\it Step 3.} {\it Proof of \eqref{lines-curves}.}

\medskip

We are now in position to prove \eqref{lines-curves}. Let $j\in\{-J_R,\ldots,J_R\}$ and $n=k+j$.
If $k\geqslant k_1+J_R$ then $n\geqslant k_1$ so that \eqref{eq:eps} holds, that is
\begin{equation*}
\abs{\psi_{k+j}(t)-((k+j)\sep+\rho_0,0)-t\vec{d}}<\eps/2
\end{equation*}
for $t\in[-T_R,T_R]$. Additionally, if $k\geqslant n_2+J_R$, $\abs{\sep_k-k\sep-\rho_0}<\eps/2$.
Altogether,
\begin{equation*}
\abs{t\vec{d}+j\sep\vec{i}-(\psi_{k+j}(t)-\sep_k \vec{i})}\leqslant\abs{t\vec{d}+((k+j)\sep+\rho_0,0)-\psi_{k+j}(t)}+\abs{k\sep+\rho_0-\sep_k}<\eps
\end{equation*}
which is exactly \eqref{lines-curves}.

\medskip

\noindent{\it Step 4.} {\it Proof of \eqref{curve-lines}.}

\medskip

It is a bit more complicated to identify the pieces of the trajectory that go through a given square. The first part of this step
consists in identifying those pieces.

We want to prove that, given $R,\eps>0$ for $k$ sufficiently large
\begin{equation*}
(\traj-(\sep_k,0))\cap(-R,R)^2\subset\tilde L_{\vec{d},\sep}+B_\eps(0).
\end{equation*}
Defining $Q_{k,R}:=(\sep_k-R,\sep_k+R)\times(-R,R)$, this can be rewritten as
\begin{equation}
\label{eq:curve-line2}\traj\cap Q_{k,R}\subset\{t\vec{d}+(\sep j+\sep_k)\vec{i}:t\in\R,j\in\Z\}+B_\eps(0).
\end{equation}
First, let $n_0$ be the smallest $n\in\N$ for which $\{\psi_n(t):t\in I_m\}\cap Q_{k,R}\neq\emptyset$. In particular $(\sep_k,0)$ is at distance at most $R$ from the trajectory of $\psi_{n_0}(t)$. As $\sep_k\to+\infty$ when $k\to\infty$, this implies that $n_0\to\infty$ when $k\to\infty$. Then, we can take $k$ large enough so that $n_0\geqslant k_\rho$, where $k_\rho$ is the constant of radial monotonicity. This means that
\begin{equation*}
\traj\cap Q_{k,R}\subset\bigcup_{n\geqslant k_\rho}\{\psi_n(t): t\in I_n\}\cap Q_{k,R},
\end{equation*}
and thus we can use the monotonicity property to order from left to right all trajectories $\{\psi_n(t):t\in I_n\}$ which intersect $Q_{k,R}$ (see Figure \ref{ConstructionStep4c}).

\begin{figure}[ht]
\begin{tikzpicture}[scale=0.5, every node/.style={scale=0.5}]
\draw[->](-3,0)--(11,0);
\draw[fill=gray,fill opacity=0.2](3,-2)--(7,-2)--(7,2)--(3,2)--(3,-2);
\draw[dashed,green](0,2)--(11,2);
\draw[dashed,green](0,-2)--(11,-2);
\node at (-0.5,2) {$R$};
\draw[purple](5.75,4)..controls(4.75,2)and(4.9,-2)..(4.25,-4);
\draw[purple](4.6,4)..controls(3.9,2)and(3.7,-2)..(3.25,-4);
\draw[purple](6.5,4)..controls(5.8,2)and(5.7,-2)..(5.1,-4);
\draw[purple](3.75,4)..controls(2.65,2)and(3,-2)..(2.4,-4);
\draw[purple](7.5,4)..controls(7.1,2)and(6.2,-2)..(6.1,-4);
\node at (5.1,-0.2) {$\sep_k$};
\end{tikzpicture}
\caption{Curves in $Q_{k,R}$ ordered from left to right.}
\label{ConstructionStep4c}
\end{figure}
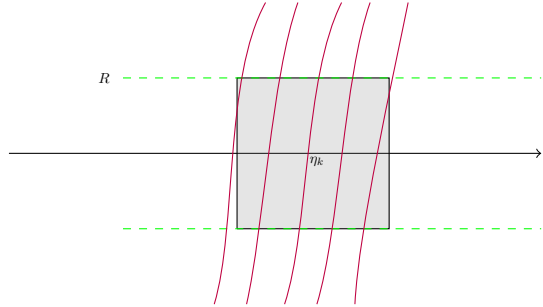
From now on, we will assume that $n\geqslant k_\rho$ for any piece of trajectory $\{\psi_n(t): t\in I_n\}$ considered.

Now, let us reduce the picture. For this, we temporarily fix $k$ and write
$z=\sep_k$, $z_\pm=z\pm R$, $\tau=z_+/z_-$. As $\sep_k\to+\infty$, we can choose $k$
so that $\tau$ is as near to $1$ as we want.

Let $\mathcal{D}_\pm=\R_+(z_-,\pm R)\cup\{(0,0)\}$ for the half-lines starting at $0$ and through the left-corners of $Q_{k,R}$ and let $S$ be the subcone delimited by these half-lines (and containing the positive $x_1$-axis).

Let $a=z_--\eps-2R\abs{d_1/d_2}$. The choice of $a$ has been made as follows: let $\ell_a$ be the line through $(a,0)$ and directed by $\vec{d}$, and set $A:=\ell_a\cap\{\abs{x_2}\leqslant 2R\}$. Then $A\subset\{x_1\leqslant z_--\eps\}$ and $A\cap\{x_1= z_--\eps\}\neq\emptyset$.

Similarly, let $b=z_++\eps+2R\abs{d_1/d_2}=z_-+\eps+2R\left(1+\abs{d_1/d_2}\right)$. Again, let $\ell_b$ be the line directed by $\vec{d}$ through $(b,0)$ and set $B:=\ell_b\cap\{\abs{x_2}\leqslant 2R\}$. Additionally, here we define $C:=\ell_b\cap S$. Then

--- $B$ is at distance at least $\eps$ on the right of $Q_{k,R}$, that is, $B\subset\{x_1\geqslant z_++\eps\}$. Moreover $B\cap\{x_1=z_++\eps\}\neq\emptyset$. Also $B\subset\{x_1\leqslant b+2R\abs{d_2/d_1}\}$ and $B\cap\{x_1=b+2R\abs{d_2/d_1}\}\neq\emptyset$. In particular, $B$ joins the boundaries of the strip $\{z_++\eps\leqslant x_1\leqslant b+2R\abs{d_2/d_1}\}$. 

--- $C$ is included in the strip $\{\frac{bd_2}{z_-d_2+Rd_1}R\leqslant x_1\leqslant \frac{bd_2}{z_-d_2-Rd_1}R\}$. Note that, when $k\to+\infty$, $b,z_-\to+\infty$ but $b/z_-\to 1$. It follows that, for $k$ large enough, $C\subset B$.

Note also that, if $k$ is large enough, then $A,B,C$ are included in the larger escape cone, see Figure \ref{ConstructionStep4}.

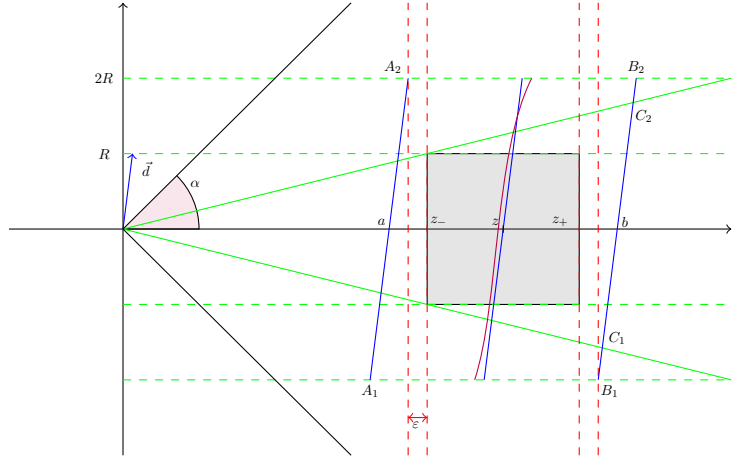
\begin{figure}[ht]
\begin{tikzpicture}[scale=0.5, every node/.style={scale=0.5}]
\draw[fill=purple,fill opacity=0.1](0,0)--(2,0) arc(0:45:2);
\node at (1.9,1.2) {$\alpha$};
\draw[->](-3,0)--(16,0);
\draw[->](0,-6)--(0,6);
\draw[fill=gray,fill opacity=0.2](8,-2)--(12,-2)--(12,2)--(8,2)--(8,-2);
\draw[green](0,0)--(16,4);
\draw[green](0,0)--(16,-4);
\draw[dashed,green](0,-4)--(16,-4);
\draw[dashed,green](0,4)--(16,4);
\draw[dashed,green](0,2)--(16,2);
\draw[dashed,green](0,-2)--(16,-2);
\node at (-0.5,2) {$R$};
\node at (-0.5,4) {$2R$};
\draw[dashed,red](7.5,-6)--(7.5,6);
\draw[dashed,red](8,-6)--(8,6);
\draw[<->,red](7.50,-5)--(8,-5);
\node at (7.7,-5.2) {$\eps$};
\draw[dashed,red](12.5,-6)--(12.5,6);
\draw[dashed,red](12,-6)--(12,6);
\draw[blue](7.5,4)--(6.5,-4);
\draw[blue](12.5,-4)--(13.5,4);
\draw[blue](10.5,4)--(9.5,-4);
\draw[purple](10.75,4)..controls(9.75,2)and(9.9,-2)..(9.25,-4);
\node at (6.5,-4.3) {$A_1$};
\node at (7.1,4.3) {$A_2$};
\node at (12.8,-4.3) {$B_1$};
\node at (13.5,4.3) {$B_2$};
\node at (13,-2.9) {$C_1$};
\node at (13.7,3) {$C_2$};
\draw[->,blue](0,0)--(0.25,2);
\node at (0.6,1.6) {$\vec{d}$};
\draw[black](10,-0.1)--(10,0.1);
\node at (9.8,0.2) {$z$};
\node at (8.3,0.2) {$z_-$};
\node at (11.5,0.2) {$z_+$};
\node at (13.2,0.2) {$b$};
\node at (6.8,0.2) {$a$};
\draw (0,0)--(6,6);
\draw (0,0)--(6,-6);
\end{tikzpicture}
\caption{The segments are $A=[A_1,A_2]$, $B=[B_1,B_2]$, $C=[C_1,C_2]$.}
\label{ConstructionStep4}
\end{figure}

Next, recall from \eqref{eq:eps} that, if $n$ is large enough, then for $\abs{t}\leqslant 2T_R$,
\begin{equation*}
\abs{\psi_n(t)-(\sep n+\rho_0,0)-t\vec{d}}<\eps/2.
\end{equation*}
In particular, as $a,b\to+\infty$ when $k\to+\infty$, if $k$ is large enough, then this holds for $a/\sep<n<b/\sep$. In other words, for those $n$'s, the part of the trajectory $\{\psi_n(t):\abs{t}\leqslant 2T_R\}$ stays at distance less than $\eps/2$ of the segments $\{(n\eta+\rho_0,0)+t\vec{d}:\abs{t}\leqslant 2T_R\}$. Those segments are all included in the rhombus delimited by $A,B$ and the lines $\{x_2=\pm 2R\}$. In particular, they all join the 2 boundaries of the cone $S$. Further, the monotonicity property shows that the curves $\{\psi_n(t):\abs{t}\leqslant 2T_R\}\cap S$ are ordered from the left to the right when $n$ goes from $a/\eta$ to $b/\eta$.

Now let $\{\psi_m(t):t\in I_m\}$ be a piece of the trajectory that intersects $Q_R$. It therefore also intersects $S$. From the monotonicity property $\{\psi_m(t):t\in I_m\}$ is either

--- on the left of $\{\psi_{\lceil a/\eta\rceil}(t):t\in I_{\lceil a/\eta\rceil}\}\cap S$ 

--- or on the right of $\{\psi_{\lfloor b/\eta\rfloor}(t):t\in I_{\lfloor b/\eta\rfloor}\}\cap S$

--- or is one of $\{\psi_n(t):t\in I_n\}\cap S$, $a/\eta< n< b/\eta$.

But, in the first two cases, $\{\psi_m(t):t\in I_m\}$ would not intersect $Q_R$ so that the only trajectories that may intersect $Q_R$ are $\{\psi_n(t):t\in I_n\}\cap S$, $a/\eta<n<b/\eta$.

Further note that, when $\abs{t}=2T_R$, \eqref{eq:eps} implies that $\abs{\psi_n(t)}\in\{\abs{x_2}\geqslant 2R-\eps/2\}\subset\{\abs{x_2}>R\}$ provided we choose $\eps<2R$. It follows that
\begin{equation*}
\traj\cap Q_{k,R}\subset\bigcup_{a/\sep<n<b/\sep}\{\psi_n(t):\abs{t}\leqslant2T_R\}.
\end{equation*}
In summary
\begin{equation*}
\traj\cap Q_{k,R}\subset\bigcup_{a/\sep<n<b/\sep}\{(\sep n+\rho_0,0)+t\vec{d}:\abs{t}\leqslant2T_R\}+B_{\eps/2}(0).
\end{equation*}
To finish, note that since $\lim_k\frac{\sep n+\sep_k}{\sep(n+k)+\rho_0}=1$ uniformly for any $n\in(a/\sep,b/\sep)$ then, for $k$ large enough,
\begin{equation*}
\traj\cap Q_{k,R}\subset\bigcup_{a/\sep<n<b/\sep}\{(\sep(n+k)+\sep_k,0)+t\vec{d}:\abs{t}\leqslant2T_R\}+B_\eps(0),
\end{equation*}
which is \eqref{eq:curve-line2}. Then \eqref{lines-curves} follows.
\end{proof}

\begin{proposition}\label{coro_xx}
Let $\traj$ be a spiraling trajectory with asymptotic velocity $\vec{d}\in\mathbb{S}^1$ and asymptotic separation $\tau>0$. Let $\Omega$ be a convex centered symmetric body such that 
\begin{equation*}
\{t\vec{d}_\perp:t\in[-1/2\tau,1/2\tau]\}\subset\Omega^\circ.
\end{equation*}
Then $\traj$ is not a sampling trajectory for $PW^2(\Omega)$.
\end{proposition}
\begin{proof}
Suppose on the contrary that $\traj$ is a sampling trajectory for $PW^2(\Omega)$. By Theorem D, there exists a separated set $\aset\subset\traj$ that is a sampling set for $PW^2(\Omega)$. By Lemma \ref{lemma_wl}, $L_{\vec{d},\tau}\in W(\traj)$. Therefore, there exists a sequence $\{x_k: k\geqslant1\}\subset\R^2$ such that $\traj+x_k\weakconv L_{\vec{d},\tau}$. By Lemma \ref{lemma_wc}, we may pass to a subsequence, and assume that $\aset+x_k\weakconv\aset'$, for some set $\aset'\subset\R^2$. Since $\aset\subset\traj$, it follows that $\aset'\subset L_{\vec{d},\tau}$.

Since $\{t\vec{d}_\perp:t\in[-1/2\tau,1/2\tau]\}\subset\Omega^\circ$, there exists $r>1/2\tau$ such that $\{t\vec{d}_\perp:t\in[-r,r]\}\subset\Omega$, and we can take $0<\eps<1$ such that
\begin{equation*}
\{t\vec{d}_\perp:t\in[-1/2\tau,1/2\tau]\}\subset\{t\vec{d}_\perp:t\in[-(1-\eps)r,(1-\eps)r]\}\subset(1-\eps)\Omega.
\end{equation*}
As $\aset$ is a sampling set for $PW^2(\Omega)$, by Corollary \ref{bodyshrink}, it is also a sampling set for $PW^\infty((1-\eps)\Omega)$. Therefore, by Theorem \ref{Beurling}, $\aset'$ is a uniqueness set for $PW^\infty((1-\eps)\Omega)$. However, the non-zero function
\begin{equation*}
f(x):=\sin\left(\frac{\pi}{\tau}\scal{x,\vec{d}_\perp}\right)
\end{equation*}
satisfies $f\equiv0$ on $L_{\vec{d},\tau}$ -- cf. \eqref{eq_equi_pl}--
and therefore on $\aset'$,
and $\supp(\widehat{f})\subset\{t\vec{d}_\perp:t\in[-1/2\tau,1/2\tau]\}\subset(1-\eps)\Omega$. This contradiction shows that $\traj$ cannot be a sampling trajectory for $PW^2(\Omega)$.
\end{proof}

\subsection{Application to concrete curves}
We now have all the elements to prove our sampling result, Theorem A, for the Archimedes spiral and the 
concentric circles. Let us recall the statement and then prove it.
\begin{thmconcrete}{\em
Let $\Omega\subset\R^2$ be a convex centered symmetric body.
\begin{enumerate}
\item If $\diam(\Omega)\sep<1$, then the Archimedes spiral $A^\sep$ and the collection of concentric circles $O^\sep$ are sampling trajectories for $PW^2(\Omega)$.

\item If $\diam(\Omega)\sep>1$, then neither the Archimedes spiral $A^\sep$ nor the collection of concentric circles $O^\sep$ are sampling trajectories for $PW^2(\Omega)$.
\end{enumerate}}
\end{thmconcrete}
\begin{proof}
Let $\traj$ be either $A^\sep$ or $O^\sep$. Note first that
\begin{equation}
\label{eq_incl}\Omega\subset\bar{B}_{\diam(\Omega)/2}(0).
\end{equation}
Indeed, if $x\in\Omega$ then by symmetry $-x\in\Omega$ and thus, $2\norm{x}=\norm{x-(-x)}\leqslant\diam(\Omega)$.

For $(i)$, assume that $\diam(\Omega)\sep<1$, and let $\eps>0$ and $\sep'>\sep$ be such that $\sep'(1+\eps)\diam(\Omega)<1$. We know from \eqref{eq_calc_gaps} that $\gap(\traj)=\sep/2$. By Lemma \ref{lemma_sep_gap}, there exists a separated set $\aset\subset\traj$ with $\gap(\aset)\leqslant\sep'/2$. Since
\begin{equation*}
\gap(\aset)\diam(\bar{B}_{(1+\eps)\diam({\Omega})/2}(0))
\leqslant \frac{\sep'}{2}(1+\eps)\diam({\Omega})<\frac{1}{2}.
\end{equation*}
Theorem \ref{sufficientdense} implies that $\aset$ is a sampling set of $PW^\infty(\bar{B}_{(1+\eps)\diam({\Omega})/2}(0))$. Finally, applying Theorem \ref{shrink}, we have that $\aset$ is a sampling set for $PW^2(\bar{B}_{\diam(\Omega)/2}(0))$ and thus for $PW^2(\Omega)$. We invoke Theorem D to conclude that $\traj$ is a sampling trajectory for $PW^2(\Omega)$.

For $(ii)$, we first note that exists $\vec{d}\in\mathbb{S}^1$ such that
\begin{equation*}
\{t\vec{d}_\perp:t\in[-\diam(\Omega)/2,\diam(\Omega)/2]\}
\end{equation*}
is contained in $\Omega$. Indeed, by compactness, we can select $x\in\Omega$ with maximal norm. Then $\Omega\subset\bar{B}_{\norm{x}}(0)$, and hence $\diam(\Omega)\leqslant2\norm{x}=\norm{x-(-x)}\leqslant\diam(\Omega)$. Thus,
$\norm{x}=\diam(\Omega)/2$. Letting $\vec{d}_\perp:=x/\norm{x}$, 
convexity reads \[\{t\vec{d}_\perp:t\in[-\diam(\Omega)/2,\diam(\Omega)/2]\}\subset\Omega.\] Now suppose that $\diam(\Omega)\sep>1$ and take $\eps\in(0,1)$ such that \[(1-\eps)^{-1} 1/(2\sep)=\diam(\Omega)/2.\] Therefore, 
\begin{equation*}
\{t\vec{d}_\perp:t\in[-1/(2\sep),1/(2\sep)]\}\subset(1-\eps)\Omega\subset\Omega^\circ
\end{equation*}
where we used \eqref{body}. By Proposition \ref{prop_ex}, $\traj$
is spiraling with asymptotic velocity $\vec{d}$ and asymptotic
separation $\tau=\sep$. We invoke Proposition \ref{coro_xx} and conclude that $\traj$ is not a sampling trajectory for $PW^2(\Omega)$.
\end{proof}

\section{Approximate aliasing}
\label{sec_al}

\subsection{Rates of convergence for weak limits}
We introduce the following class of curves.

\begin{definition}\label{class}
Let $\sep>0$, $\lambda>0$ and $\Omega\subset\R^2$ a convex centered symmetric body. We say that a regular trajectory $\traj$ belongs to the class $\Cc(\sep,\lambda,\Omega)$ if there exists a constant $C_\traj$ for which, given $\eps,R>0$ there exists $y\in\mathbb{R}^2$ such that
\begin{enumerate}
\item $(\traj-y)\cap(-R,R)^2\subset(\sep\mathbb{Z}\times\mathbb{R})+B_{\sep\eps}(0)$,

\item $\abs{y}\leqslant C_\traj\frac{R^2}{\eps}$,

\item $\lambda^{-1/p}\norm{f}_{L^p(\mu_\traj)}\leqslant\norm{f}_p$ for all $f\in PW^p(\Omega)$ with $1\leqslant p\leqslant\infty$.
\end{enumerate}
\end{definition}
We now prove that the curves $A^\sep$ and $O^\sep$ belong to this type of classes. For technical reasons we extend this result to any rotation $\Rr_{2\pi\theta_0}A^\sep$. As a first step, we show the following lemma, which quantifies the convergence in Lemma \ref{lemma_wl}.

\begin{lemma}\label{rate-spiral}
Let $\sep>\sqrt{2}/2$ and $\theta_0\in[0,1)$. Then there exists a constant $C$ independent from $\sep$ and $\theta_0$ such that given $R\geqslant1$ and $0<\eps<1/2$ the inclusion
\begin{equation*}
(\Rr_{2\pi\theta_0}A^\sep-(\sep(n+\theta_0),0))\cap(-R,R)^2\subset\sep\Z\times\R+B_\eps(0)
\end{equation*}
holds for all $n\geqslant C\eps^{-1}R^2$. The same inclusion holds for $O^\sep$, translating instead by $(\sep n,0)$ and eventually taking a larger constant.
\end{lemma}
\begin{proof}
Let $x=(\sep\theta\cos2\pi(\theta-\theta_0),\sep\theta\sin2\pi(\theta-\theta_0))\in \Rr_{2\pi\theta_0}A^\sep\cap(\R\times(-R,R))$ with $x_1>\sep(\theta_0+1/4)$. Then, there exists $k\in\N$ such that $k-1/4<\theta-\theta_0<k+1/4$. We will prove
\begin{equation}
\label{x1}\abs{x_1-\sep(k+\theta_0)}\leqslant\frac{33R^2}{k}.
\end{equation}
Assuming this for a moment, let us show how the lemma would then follow. Let $n\geqslant 2R$ and take $y\in(\Rr_{2\pi\theta_0}A^\sep-(\sep(n+\theta_0),0))\cap(-R,R)^2$. Write $y=x-(\sep(n+\theta_0),0))$ with $x\in\Rr_{2\pi\theta_0}A^\sep$. Then $x_1>\sep(n+\theta_0)-R$, and since $n\geqslant2R$, $R\geqslant1$ and $2\sep>1$, we have $x_1>\sep(\theta_0+1/4)$. Further, taking the same $k$ as before,
\begin{equation*}
\sep(n+\theta_0)-R<x_1=\sep\theta\cos2\pi(\theta-\theta_0)\leqslant\sep\theta\abs{\cos2\pi(\theta-\theta_0)}\leqslant\sep\theta\leqslant\sep(k+1/4+\theta_0).
\end{equation*}
This yields 
\begin{equation*}
k>n-1/4-R/\sep>n/2-R^2/\eps
\end{equation*}
since $n\geqslant1$ and $\sep>1/2>\eps$. Hence, taking $n\geqslant\frac{68R^2}{\eps}\geqslant2R$ we get $k>\frac{33R^2}{\eps}$ and then \eqref{x1} reads $\abs{y_1-\sep(n-k)}=\abs{x_1-\sep(k+\theta_0)}<\eps$. Therefore, $y\in\sep\Z\times\R+B_\eps(0)$, as claimed.

Let us now prove \eqref{x1}. Using the same notation as before for $x\in\Rr_{2\pi\theta_0}A^\sep\cap(\R\times(-R,R))$ with $x_1>\sep(\theta_0+1/4)$, we have
\begin{gather*}
k-1/4<\theta-\theta_0<k+1/4,
\\
-R<\sep\theta\sin2\pi(\theta-\theta_0)<R.
\end{gather*}
In particular, since $\arcsin(\theta)$ is a strictly increasing function and $\abs{\arcsin(\theta)}\leqslant2\abs{\theta}$ for $\theta\in[-1,1]$,
\begin{equation*}
\abs{\theta-(k+\theta_0)}<\frac{1}{2\pi}\cdot\frac{2R}{\sep\theta}<\frac{R}{\sep\pi(k-1/4+\theta_0)}<\frac{2R}{\sep\pi k}
\end{equation*}
since $\theta_0\geqslant0$, $k\geqslant1$. Next, using the triangle inequality and the fact that $\abs{\cos\theta-1}\leqslant\abs{\theta}^2/2$ we get
\begin{align*}
\abs{x_1-\sep(k+\theta_0)}&=\abs{\sep\theta\cos2\pi(\theta-\theta_0)-\sep(k+\theta_0)}
\\
&\leqslant\abs{\sep\theta\cos2\pi(\theta-\theta_0)-\sep(k+\theta_0)\cos2\pi(\theta-\theta_0)}
\\
&\qquad+\abs{\sep(k+\theta_0)\cos2\pi(\theta-\theta_0)-\sep(k+\theta_0)}
\\
&\leqslant\sep\abs{\theta-(k+\theta_0)}+\sep(k+\theta_0)\abs{\cos2\pi(\theta-k-\theta_0)-1}
\\
&\leqslant\sep\abs{\theta-(k+\theta_0)}+2\pi^2\sep(k+\theta_0)\abs{\theta-(k+\theta_0)}^2
\\
&\leqslant\frac{2R}{\pi k}+4\pi^2\sep k\frac{4R^2}{(\sep\pi k)^2}\leqslant\frac{R}{k}+\frac{16R^2}{\sep k}\leqslant\frac{33R^2}{k},
\end{align*}
since $2\sep,R,k\geqslant1>\theta_0$. The proof for $O^\sep$ is similar.
\end{proof}

\begin{proposition}\label{class-spiral}
Let $\sep>\sqrt{2}/2$, $\theta_0\in[0,1)$, and $R_0 >0$.
Then, there exists a constant $C>0$ independent of $\sep$ and $\theta_0$ such that the curves 
$\Rr_{2\pi\theta_0}A^\sep$ and $O^\sep$ belong to the class $\Cc(\sep,C(R_0+1),\bar{B}_{R_0}(0))$.
Moreover, the constant $C_\curve=C_{\Rr_{2\pi\theta_0}A^\sep}, C_{O^\sep}$ in Definition \ref{class}
is also independent of $\sep$, $\theta_0$.
\end{proposition}

\begin{remark}
The proof below also works for $0<\sep\leqslant\sqrt{2}/2$ but then the constants $C_{\Rr_{2\pi\theta_0}A^\sep}$ and $C_{O^\sep}$ depend on $\sep$. The reader may check that they satisfy a bound of the form $C\sep^{-1}$.
\end{remark}

\begin{proof}[Proof of Proposition \ref{class-spiral}]
We treat only $\Rr_{2\pi\theta_0}A^\sep$, the case of $O^\sep$ being similar. 
Condition 
$(iii)$ of Definition \ref{class} follows from Proposition \ref{JM}: since $\sep>\sqrt{2}/2$, for all 
$f\in PW(\bar{B}_{R_0}(0))$
\begin{equation*}
\norm{f}_{L^p(\mu_\traj)}\lesssim(\sep^{-1/p}+R_{0}^{1/p})\norm{f}_p\lesssim(R_0+1)^{1/p}\norm{f}_p.
\end{equation*}
Conditions $(i)$ and $(ii)$ follow from the Lemma \ref{rate-spiral}. Indeed, 
let $\eps,R>0$ be given and set $\eps':=\sep\eps$. Without loss of generality, we can 
assume that $\eps'\in(0,1/2)$ and $R\geqslant1$. Taking $y=(\sep(n+\theta_0),0)$ with $n=\lceil C(\sep\eps')^{-1}R^2\rceil$, Lemma \ref{rate-spiral} gives
\begin{gather*}
(\Rr_{2\pi\theta_0}A^\sep-y)\cap(-R,R)^2\subset\sep\Z\times\R+B_{\sep\eps}(0),
\\
\abs{y}\leqslant2C\eps^{-1}R^2,
\end{gather*}
as desired.
\end{proof}

\subsection{Quantitative aliasing}
Having quantified the convergence in Lemma \ref{lemma_wl}, we turn into the quantification of aliasing. While a union of lines with sub-Nyquist density leads to aliasing, we show that spirals suffer from approximate aliasing.

Let us recall the definition of {\it variation} of a function $f\in L^1(\Omega)$:
\begin{equation}
\label{eq_def_var}
\var_\Omega(f):=\sup\left\{
\big|{
\int_\Omega f\text{div}h}\big|:h\in C_c^1(\Omega, 
\mathbb{C}^2)\text{, }\norm{h}_\infty\leqslant1\right\}.
\end{equation}
When $f$ belongs to the Sobolev space $W^{1,1}(\Omega)$ integration by parts shows that $\var_\Omega(f)=\norm{\nabla f}_{L^1(\Omega)}$. 

\begin{lemma}\label{fhat}
Let $\sep>\sqrt{2}/2$, $\lambda>0$ and consider $\traj\in\Cc(\sep,\sep\lambda,\bar{B}_2(0))$. Let
$Q$ be the square of vertices $(0,0)$, $(\sqrt{2}/4,\sqrt{2}/4)$, $(\sqrt{2}/2,0)$ and $(\sqrt{2}/4,-\sqrt{2}/4)$. Then given $\zeta>0$, there exists $g\in PW^2(Q\cup-Q)$ such that
\begin{enumerate}
\item $\norm{g}_2=1$,

\item $\sep^{-1/2}\norm{g}_{L^2(\mu_\traj)}\leqslant\zeta$,

\item $\var(\widehat{g})\leqslant C\left(\max((\sep-\sqrt{2}/2)^{-1},\sep)\frac{\lambda}{\zeta^2}\ln^4\left(C\frac{\lambda}{\zeta^2}\right)+1\right)$,

\item $\norm{\widehat{g}}_\infty\leqslant C\max((\sep-\sqrt{2}/2)^{-1},\sep)$,
\end{enumerate}
where $C>0$ is a universal constant.
\end{lemma}
\begin{proof}
Let $Q_0$ be the square of vertices $(-1,0)$, $(0,1)$, $(1,0)$ and $(0,-1)$. Take $\phi\in PW^2(Q_0)$ such that $\norm{\phi}_2=\sqrt{2}/2$ and
\begin{equation}
\label{phiexp}\abs{\phi(x)}\leqslant C_1e^{-\abs{x}^{1/2}}
\end{equation}
for some constant $C_1>0$. Note that, integrating in polar coordinates
\begin{equation*}
\norm{\phi}_1\leqslant2\pi C_1\int_0^\infty e^{-r^{1/2}}r\d r=4\pi\int_0^\infty e^{-s}s^3\d s=24\pi C_1.
\end{equation*}
Also, as $\phi\in PW^2(Q_0)$ and the area of $Q_0$ is $2$ and $\norm{\widehat{\phi}}_2=\norm{\phi}_2=\sqrt{2}/2$, then Cauchy-Schwartz inequality gives $\norm{\widehat{\phi}}_1\leqslant\abs{Q_0}^{1/2}\norm{\widehat{\phi}}_2=1$.

Define $g_0(x):=\beta\sin\left(\frac{\pi}{\sep}x_1\right)\phi(\beta x)$ with $\beta=\min(1/(2\sep),\sqrt{2}/2-1/(2\sep))$. Note that, as $\sep>\sqrt{2}/2$, $\beta>0$. For $\eps,R>0$ to be fixed later take $y=y(\eps,R)$ as in Definition \ref{class} and set $g(x):=g_0(x-y)$. Then
\begin{multline}
\label{plancha}\widehat{g}(\xi)=\frac{e^{-2\pi iy\xi}}{2\beta i}\left[\widehat{\phi}(\beta^{-1}(\xi_1-1/(2\sep)),\beta^{-1}\xi_2)-\widehat{\phi}(\beta^{-1}(\xi_1+1/(2\sep)),\beta^{-1}\xi_2)\right]
\\
=:\phi_1(\xi)-\phi_2(\xi).
\end{multline}
As $\phi\in PW^2(Q_0)$, the support of $\phi_1$ is included in $Q_\sep:=\beta Q_0+(1/(2\sep),0)$ and the 
support of $\phi_2$ in $-Q_\sep$. We claim that $Q_\sep\subset Q$. 
To show this, we argue by cases 
on $1/(2\sep)$. If $1/(2\sep)\leqslant\sqrt{2}/4$ this means that $\beta=1/(2\sep)$ and hence the points 
$(1/(2\sep),\beta)$ and $(1/(2\sep),-\beta)$ belong to the square $Q$. Also since $1/(2\sep)-\beta=0$ and 
$1/(2\sep)+\beta\leqslant\sqrt{2}/2$, both $(1/(2\sep)-\beta,0)$ and $(1/(2\sep)+\beta,0)$ belong to $Q$. By 
convexity this yields $Q_\sep\subset Q$. If $1/(2\sep)\geqslant\sqrt{2}/4$ this means that 
$\beta=\sqrt{2}/2-1/(2\sep)$ and hence the points $(1/(2\sep),\beta)$ and $(1/(2\sep),-\beta)$ belong to 
$Q$.
Also in this case $1/(2\sep)-\beta\geqslant0$ and $1/(2\sep)+\beta=\sqrt{2}/2$, so that, again, 
$(1/(2\sep)-\beta,0)$ and $(1/(2\sep)+\beta,0)$ belong to $Q$. Thus $Q_\sep\subset Q$ also in this case. 
Thus, 
in any case, $\supp(\phi_1)\subset Q$ and $\supp(\phi_2)\subset -Q$, where $Q$ is the cube 
defined in the hypothesis. Altogether, $g\in PW^2(Q\cup-Q)$.
\begin{figure}[ht]
\begin{tikzpicture}[scale=0.5, every node/.style={scale=0.5}]
\draw[->](0,0)--(10,0)
node[pos=0.3,below]{$1/(2\sep)$};
\draw[|-|,blue](3,0)--(6,0);
\node at (4.5,0.2) {$\beta$};
\draw[->](0,-6)--(0,6);
\draw[fill=gray,fill opacity=0.2](3,3)--(6,0)--(3,-3)--(0,0);
\draw[fill=gray,fill opacity=0.1](0,0)--(4,4)--(8,0)--(4,-4)--(0,0);
\draw(4.4,2)node{$Q_\sep$};
\node at (4.4,4) {$Q$};
\end{tikzpicture}
\hspace*{5em}
\begin{tikzpicture}[scale=0.5, every node/.style={scale=0.5}]
\draw[->](0,0)--(10,0)
node[pos=0.5,below]{$1/(2\sep)$};
\draw[|-|,blue](2,0)--(5,0);
\node at (3.5,0.2) {$\beta$};
\draw[->](0,-6)--(0,6);
\draw[fill=gray,fill opacity=0.2](8,0)--(5,3)--(2,0)--(5,-3);
\draw[fill=gray,fill opacity=0.1](0,0)--(4,4)--(8,0)--(4,-4)--(0,0);
\draw(3.4,2)node{$Q_\sep$};
\node at (4.4,4) {$Q$};
\end{tikzpicture}
\caption{Sketch of both cases: $1/(2\sep)\leqslant\sqrt{2}/4$ (left), $1/(2\sep)\geqslant\sqrt{2}/4$ (right).}
\end{figure}
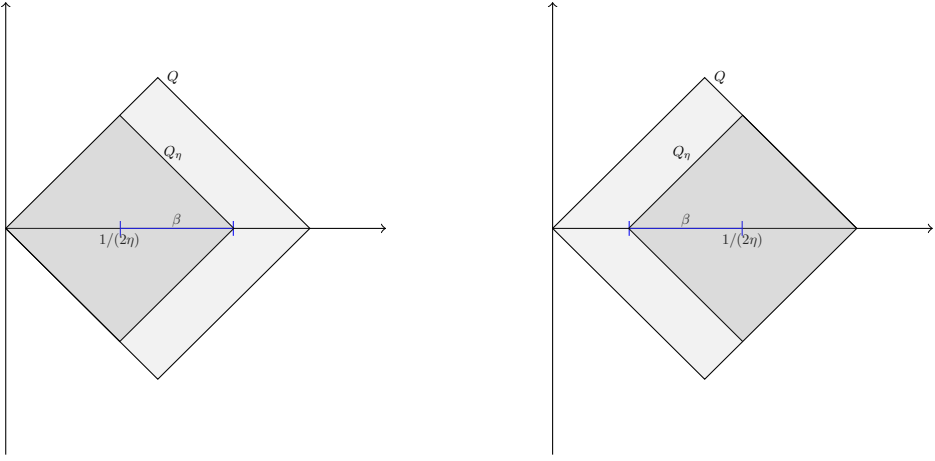

Since $Q$ and $-Q$ only intersect at $0$, then $\phi_1$ and $\phi_2$ have disjoint support up to a set of measure zero. Hence, using Plancharel we obtain
\begin{equation*}
\norm{g}_2^2=\norm{\widehat{g}}_2^2=\norm{\phi_1}_2^2+\norm{\phi_2}_2^2=2\norm{\phi}_2^2=1,
\end{equation*}
which proves $(i)$.

We now show that $\sep^{-1}\norm{g}_{L^2(\mu_\traj)}$ is small enough for an appropriate choice of $\eps$ and $R$. First note that $\sup_{x\in\sep\mathbb{Z}\times\mathbb{R}+\bar{B}_{\sep\eps}(0)}\abs{g_0(x)}\leqslant\beta\eps\norm{\phi}_{L^\infty(\R^2)}\leqslant\beta\eps C_1$. Then, combining this with condition $(i)$ of Definition \ref{class} and \eqref{phiexp} we have
\begin{align*}
\sup_{x\in\traj}\abs{g(x)}=\sup_{x\in\traj-y}\abs{g_0(x)}&\leqslant\sup_{x\in(\traj-y)\cap(-R,R)^2}\abs{g_0(x)}+\sup_{x\notin B_R(0)}\abs{g_0(x)}
\\
&\leqslant\sup_{x\in\sep\mathbb{Z}\times\mathbb{R}+\bar{B}_{\sep\eps}(0)}\abs{g_0(x)}+\beta C_1e^{-\beta^{1/2}R^{1/2}}
\\
&\leqslant\beta C_1(\eps+e^{-\beta^{1/2}R^{1/2}}).
\end{align*}
Since $Q\cup-Q\subset\bar{B}_2(0)$, we can use the condition $(iii)$ of Definition \ref{class} together with interpolation to get
\begin{align*}
\sep^{-1}\norm{g}^2_{L^2(\mu_\traj)}&\leqslant\sep^{-1}\norm{g}_{L^1(\mu)}\sup_{x\in\traj}\abs{g(x)}
\\
&\leqslant\lambda\norm{g}_1\sup_{x\in\traj}\abs{g(x)}=\lambda\norm{g_0}_1\sup_{x\in\traj}\abs{g(x)}
\\
&\leqslant\frac{\lambda\norm{\phi}_1}{\beta}\sup_{x\in\traj}\abs{g(x)}\leqslant24\pi\lambda
C_1^2(\eps+e^{-\beta^{1/2}R^{1/2}}).
\end{align*}

And finally taking
\begin{equation*}
\eps=\frac{\zeta^2}{48\pi\lambda C_1^2},\qquad R=\beta^{-1}\ln^2(48\pi\lambda C_1^2/\zeta^2),
\end{equation*}
we conclude that $\sep^{-1/2}\norm{g}_{L^2(\mu_\traj)}\leqslant\zeta$, which is point $(ii)$.

Note that
\begin{equation*}
\frac{R^2}{\eps}=C_2\frac{\lambda}{\zeta^2\beta^2}\ln^4(C_2\lambda/\zeta^2)
\end{equation*}
with $C_2=48\pi C_1^2$.

To estimate the variation of $\widehat{g}$ we use again \eqref{plancha} and the fact that $\widehat{\phi}$ is smooth and rapidly decreasing combined with Leibniz's rule:
\begin{equation*}
\var(\widehat{g})=\norm{\nabla\widehat{g}}_1\leqslant\pi\abs{y}\beta\norm{\widehat{\phi}}_1+\norm{\nabla\widehat{\phi}}_1\leqslant C_3\frac{\lambda}{\zeta^2\beta}\ln^4(C_2\lambda/\zeta^2)+C_4
\end{equation*}
where $C_3=\pi C_2C_\traj$ and $C_4=\norm{\nabla\widehat{\phi}}_1$. Also, \eqref{plancha} implies 
\begin{equation*}
\norm{\widehat{g}}_\infty\leqslant\beta^{-1}\norm{\widehat{\phi}}_\infty.
\end{equation*}
Finally we see that $\beta^{-1}\leqslant2\max((\sep-\sqrt{2}/2)^{-1},\sep)$ and thus parts $(iii)$ and $(iv)$ follow by taking $C=\max(C_2,2C_3,C_4,2\norm{\widehat{\phi}}_\infty)$.
\end{proof}

\begin{proposition}\label{prop_cube}
Let $\sep=(1+\eps)\sqrt{2}/2$ with $\eps\in(0,1)$, and $\traj=A^\sep$ or $\traj=O^\sep$. Then given $\zeta>0$, there exists $f\in L^2([-1/2,1/2]^2)$ such that
\begin{enumerate}
\item $\norm{f}_2=1$,

\item $\sep^{-1/2}\norm{\widehat{f}}_{L^2(\mu_\traj)}\leqslant\zeta$,

\item $\var(f)\leqslant C(\eps^{-1}\zeta^{-2}\ln^4(C\zeta^{-2})+1)$,

\item $\norm{f}_\infty\leqslant C\eps^{-1}$,
\end{enumerate}
where $C>0$ is a universal constant.
\end{proposition}
\begin{proof}
By Proposition \ref{class-spiral}, $\Rr_{\pi/4}A^\sep\in\Cc(\sep,C,\bar{B}_2(0))$ for some constant $C>0$. Set $\lambda:=C\sep^{-1}$ so that $\Rr_{\pi/4}A^\sep\in\Cc(\sep,\sep\lambda,\bar{B}_2(0))$. Then using Lemma 
\ref{fhat}, we can construct $g\in L^2(\R^2)$ associated with $\Rr_{\pi/4}A^\sep$ and a given constant $\zeta>0$. Define $f:=\widehat{g}\circ\Rr_{-\pi/4}$. We will prove that $f$ satisfies the conditions $(i)-(iv)$ of the Proposition.

As in Lemma \ref{fhat}, let $Q$ be the cube defined by the vertices $(0,0)$, $(\sqrt{2}/4,\sqrt{2}/4)$, $(\sqrt{2}/2,0)$ and $(\sqrt{2}/4,-\sqrt{2}/4)$. Since $\supp(\widehat{g})\subset Q\cup-Q$ and
\begin{equation*}
\Rr_{-\pi/4}^{-1}(Q)=\Rr_{\pi/4}(Q)=[0,1/2]\times[0,-1/2],\quad \Rr_{-\pi/4}^{-1}(-Q)=[-1/2,0]\times[0,1/2],
\end{equation*}
we have
\begin{equation*}
\supp(f)\subset[0,1/2]\times[0,-1/2]\cup[-1/2,0]\times[0,1/2]\subset[-1/2,1/2]^2.
\end{equation*}
Point $(i)$ follows from the fact that rotations are norm-invariant transformations and $\norm{\widehat{g}}_2=1$. To see $(ii)$ note that $\widehat{f}=g\circ\Rr_{-\pi/4}^T$ and $\Rr_{-\pi/4}^T(A^\sep)=\Rr_{-\pi/4}^{-1}(A^\sep)=\Rr_{\pi/4}A^\sep$.

The variation of $f$ can be estimated by the chain rule and $(iii)$ of Lemma \ref{fhat}:
\begin{align*}
\var(f)&=\norm{\nabla f}_1=\norm{\nabla(\widehat{g}\circ\Rr_{-\pi/4})}_1\leqslant\norm{(\nabla\widehat{g})\circ\Rr_{-\pi/4}}_1\norm{\nabla \Rr_{-\pi/4}}_\infty
\\
&\lesssim\norm{\nabla\widehat{g}}_1=\var(\widehat{g})\lesssim\max((\sep-\sqrt{2}/2)^{-1},\sep)\frac{\lambda}{\zeta^2}\ln^4 \left(C\frac{\lambda}{\zeta^2}\right)+1.
\end{align*}
Then $(iii)$ follows by using the identities $\sep=(1+\eps)\sqrt{2}/2$ and $\lambda=C\sep^{-1}$ so that 
$\max((\sep-\sqrt{2}/2)^{-1},\sep)=(\sep-\sqrt{2}/2)^{-1}=\eps^{-1}$ and 
$\lambda\asymp\sep^{-1}\asymp\sep\asymp 1$. Since $\norm{f}_\infty=\norm{\widehat{g}}_\infty$, this 
argument also yields $(iv)$.
\end{proof}

\section{Compressibility and sub-Nyquist sampling}
\subsection{Analog compressibility}
We will now work on the cube $[-1/2,1/2]^2$. We recall the class $\Ff(W)$ defined in the introduction:
\begin{equation*}
\Ff(W):=\{f\in L^2([-1/2,1/2]^2):\var(f)\leqslant W\}.
\end{equation*}
The relevant {\it stability margin} is
\begin{equation*}
A(\traj,\Ff(W)):=\inf\{\norm{\hat{f}}_{L^2(\mu_\traj)}:\norm{f}_2=1\text{, }f\in\Ff(W)\}.
\end{equation*}
Let us now restate and prove Theorem B.
\begin{thmbv}
Let $\sep=(1+\eps)\sqrt{2}/2$ with $\eps\in(0,1)$, and $\traj=A^\sep$ or $\traj=O^\sep$. Then for $W>0$,
\begin{equation}
\label{eq_qwerty1}A(\traj,\Ff(W))\leqslant K (\eps W)^{-1/2}(\ln^2(\eps W)+1),
\end{equation}
where $K>0$ is a universal constant.
\end{thmbv}
\begin{proof}
Let $T>0$, set $\zeta:=\eps^{-1/2}T^{-1/2}$ and take $f$ given by Proposition \ref{prop_cube} associated to $\zeta$. Then there exists a constant $C>0$ such that
\begin{equation*}
\var(f) \leqslant \Phi(T),
\end{equation*}
where 
\begin{equation*}
\Phi(T) := C (T\ln^4(C\eps T)+1).
\end{equation*}
Since $\norm{f}_2=1$, $\sep^{-1/2}\norm{\widehat{f}}_{L^2(\mu_\traj)}\leqslant\zeta$ and 
$\sep\asymp1$, we have that
\begin{equation}
\label{eq_qwerty2}
A(\traj,\Ff(\Phi(T)))\leqslant\norm{\widehat{f}}_{L^2(\mu_\traj)}\leqslant\sep^{1/2}\zeta\lesssim\eps^{-1/2}T^{-1/2}.
\end{equation}
We claim that
\begin{equation}
\label{eq_qwerty3}
T\gtrsim\frac{\Phi(T)}{\ln^4(\eps\Phi(T))+1},\quad\text{if}\quad T\geqslant C^{-1}\eps^{-1}e.
\end{equation}
Indeed, if $C\eps T\geqslant e$, then, since $\eps\in(0,1)$,
\begin{equation*}
CT\leqslant\Phi(T)\quad\text{and}\quad T\geqslant e/C.
\end{equation*}
Consequently,
\begin{equation*}
\Phi(T)\lesssim T\ln^4(C\eps T)+1\lesssim T\ln^4(\eps\Phi(T))+T
\end{equation*}
from which \eqref{eq_qwerty3} follows.
Combining \eqref{eq_qwerty2} and \eqref{eq_qwerty3}, we conclude that
\begin{equation}
\label{eq_qwerty4}A(\traj,\Ff(\Phi(T)))\lesssim\eps^{-1/2}\Phi(T)^{-1/2}(\ln^2(\eps\Phi(T))+1),
\end{equation}
provided that $T\geqslant C^{-1}\eps^{-1}e$.

Note that $\lim_{T \to +\infty}\Phi(T)=+\infty$, while $\Phi(C^{-1}\eps^{-1}e)=\eps^{-1}e+C$. This means that any number $W\geqslant\eps^{-1} e +C$ can be represented as $W=\Phi(T)$, for some $T\geqslant C^{-1}\eps^{-1}e$. Therefore, \eqref{eq_qwerty4} reads \eqref{eq_qwerty1} for all $W\geqslant\eps^{-1}e+C$.

Finally, if $W\leqslant\eps^{-1}e+C$, then $\eps W\lesssim 1$, and the right-hand side of \eqref{eq_qwerty1} is $\gtrsim 1$. On the other hand, the Bessel bound in Proposition \ref{JM}
implies that
\begin{equation*}
A(\traj,\Ff(W))\leqslant\norm{\widehat{f}}_{L^2(\mu_\traj)}\lesssim(\sep^{-1/2}+(1/2)^{1/2})\norm{\widehat{f}}_2\lesssim\norm{f}_2\lesssim1,
\end{equation*}
since $\sep\asymp1$. This completes the proof.
\end{proof}

\subsection{Sampling wavelet-sparse signals}
We work with the Haar basis in $L^2([-1/2,1/2]^2)$ constructed from the one in $L^2([0,1])$ by
translation and tensorization: from $h^0=\chi_{[0,1)}$ and $h^1=\chi_{[0,1/2)}-\chi_{[1/2,1)}$ one defines
\begin{equation*}
h^e_{j,k}(x)=2^{j/2}h^{e_1}(2^j(x_1+1/2)-k_1)2^{j/2}h^{e_2}(2^j(x_2+1/2)-k_2)
\end{equation*}
with $j\geqslant0$, $(k_1,k_2)\in\Z^2\cap 2^j[0,1)^2$, $e\in\{0,1\}^2\setminus\{(0,0)\}$. We denote $I$ the set of all such triples, and then define the sparsity classes as
\begin{equation*}
\Sigma_{N,J}:=\left\{\sum_{(j,k,e)\in I}c^e_{j,k}h^e_{j,k}:c^e_{j,k}\in\C,\#I\leqslant N\text{, }0\leqslant j\leqslant J\right\}
\end{equation*}
and the corresponding stability margin
\begin{equation*}
A_{N,J}(\traj):=\inf\left\{\norm{\widehat{f}}_{L^2(\mu_\traj)}:\norm{f}_2=1,f\in\Sigma_{N,J} \right\}.
\end{equation*}
Let us now restate and prove Theorem C.
\begin{thmwav}
Let $\sep=(1+\eps)\sqrt{2}/2$ with $\eps\in(0,1)$, and $\traj=A^\sep$ or $\traj=O^\sep$. Then for $N \geqslant 1$, 
\begin{equation}
\label{eq_anj}
A_{N,J}(\traj)\leqslant KN^{-1/6}\eps^{-1}\ln^4(K N^{1/3}),
\end{equation}
where $J=K\ln(\eps^{-1}N)$, and $K>0$ is a universal constant.
\end{thmwav}
\begin{proof}
Fix $N$ and let $\zeta=N^{-1/6}$ and take $f\in L^2([-1/2,1/2]^2)$ as in Proposition \ref{prop_cube}. Then 
\begin{gather}
\label{2}\norm{f}_2=1\quad\text{and}\quad\norm{\widehat{f}}_{L^2(\mu_\traj)}\lesssim N^{-1/6},
\\
\label{var}\var(f)\lesssim L:=\eps^{-1}N^{1/3}\ln^4(CN^{1/3})+1\quad\text{and}\quad\norm{f}_\infty\lesssim\eps^{-1}.
\end{gather}

\noindent {\it Step 1}. Let us consider the class
\begin{equation*}
\Sigma_N:=\left\{\sum_{(j,k,e)\in I}c^e_{j,k}h^e_{j,k}:c^e_{j,k}\in\C,\#I\leqslant N\right\}
\end{equation*}
of functions with at most $N$ active Haar coefficients {\it without} restrictions on the scale. Let 
$f_N$ the best approximation of $f$ in $\Sigma_N$. Since the variation of $f$ is bounded, we can use the following inequality from \cite[Theorem 8.2]{MR1738406}:
\begin{equation}
\label{f_N}\norm{f-f_N}_2\leqslant KN^{-1/2}\var(f)\lesssim N^{-1/2}L,
\end{equation}
where $K>0$ is some universal constant. Let $P_J$ be the orthogonal projection onto the span of wavelets with $0\leqslant j \leqslant J$. Then $\norm{f-P_Jf}_1\leqslant 2^{-J}\var(f)$ - see for example \cite[Eq. 2.10]{MR1738406} and the references therein, or \cite{MR1261635}. In addition, 
$\norm{f-P_Jf}_\infty\lesssim\norm{f}_\infty$, and therefore interpolation yields
\begin{equation*}
\norm{f-P_Jf}_2\lesssim2^{-J/2}\var(f)^{1/2}\norm{f}^{1/2}_\infty.
\end{equation*}
Plugging \eqref{var} we get
\begin{equation}
\label{P_Jf}\norm{f-P_Jf}_2\lesssim2^{-J/2}\eps^{-1}L^{1/2}\leqslant2^{-J/2}\eps^{-1}L
\end{equation}
where we have also used the fact that $L\geqslant1>\eps$. Now set $f_{N,J}:=P_Jf_N \in\Sigma_{N,J}$, 
and combine \eqref{P_Jf} and \eqref{f_N} to obtain
\begin{multline*}
\norm{f-f_{N,J}}_2\leqslant\norm{f-P_Jf}_2+\norm{P_Jf-P_Jf_N}_2
\\
\leqslant\norm{f-P_Jf}_2+\norm{f-f_N}_2\lesssim(2^{-J/2}\eps^{-1/2}+N^{-1/2})L.
\end{multline*}
Hence, choosing $J\asymp\ln(\eps^{-1}N)$,
\begin{equation}
\label{f_NJ}\norm{f-f_{N,J}}_2\lesssim N^{-1/2}L.
\end{equation}

\noindent {\it Step 2}. Write $A_{N,J}=A_{N,J}(\traj)$. Then, using \eqref{2} and the definition of $A_{N,J}$,
\begin{align}
\nonumber
A_{N,J}=A_{N,J}\norm{f}_2&\leqslant A_{N,J}\norm{f_{N,J}}_2+A_{N,J}\norm{f_{N,J}-f}_2
\\
\nonumber&\leqslant\norm{\widehat{f_{N,J}}}_{L^2(\mu_\traj)}+A_{N,J}\norm{f_{N,J}-f} _2
\\
\nonumber&\leqslant\norm{\widehat{f}}_{L^2(\mu_\traj)}+\norm{\widehat{f_{N,J}}-\widehat{f}}_{L^2(\mu_\traj)}+A_{N,J}\norm{f_{N,J}-f}_2
\\
&\leqslant K_0N^{-1/6}+K_0\norm{f_{N,J}-f}_2+A_{N,J}\norm{f_{N,J}-f}_2,\label{AN}
\end{align}
for some constant $K_0>0$ and where we have also applied Proposition \ref{JM} to estimate $\norm{\widehat{f_{N,J}}-\widehat{f}}_{L^2(\mu_\traj)}$. Now from \eqref{f_NJ}, if
\begin{equation*}
N\gtrsim L^2,
\end{equation*}
and the implicit constant is large enough, then $\norm{f-f_{N,J}}_2<1/2$. Going back to \eqref{AN} and re-applying \eqref{f_NJ} we get
\begin{equation*}
A_{N,J}\lesssim N^{-1/6}+N^{-1/2}L\lesssim N^{-1/6}\eps^{-1}\ln^4(C N^{1/3}).
\end{equation*}
Note that since $L=\eps^{-1}N^{1/3}\ln^4(CN^{1/3})+1$, then, for $N\gtrsim L^2$ to hold it is sufficient to have $N\gtrsim\eps^{-2}N^{2/3}\ln^8(CN^{1/3})$, or, equivalently,
\begin{equation*}
N\gtrsim\eps^{-6}\ln^{24}(CN^{1/3}).
\end{equation*}
Hence, there is a constant $C'>0$ such that \eqref{eq_anj} holds if
$N\geqslant C'\eps^{-6}\ln^{24}(C'N^{1/3})$. On the other hand, if $N\leqslant C'\eps^{-6}\ln^{24}(C'N^{1/3})$ then to prove \eqref{eq_anj} is enough to show $A_{N,J}\lesssim1$, which, as in the proof of Theorem B, follows from the Bessel bounds in Proposition \ref{JM}.
\end{proof}

\section*{Acknowledgments}
Ph. J. kindly acknowledges financial support from the French ANR program, ANR-12-BS01-0001 (Aventures), the French-Tunisian CMCU/UTIQUE project 32701UB Popart. This study has been carried out with financial support from the French State, managed by the French National Research Agency (ANR) in the frame of the Investments for the Future Program IdEx Bordeaux - CPU (ANR-10-IDEX-03-02). F.N. is supported by the doctoral grant POS-CFRA-2015-1-125008 of Agencia Nacional de Innovación e Investigación (Uruguay) and Campus France (France). F.N. would like to thank Andreas Hartmann and Karim Kellay for helpful discussions about sampling measures.  J. L. R. gratefully acknowledges support from the Austrian Science Fund (FWF):P 29462-N35, and from the WWTF grant INSIGHT (MA16-053). The three authors kindly acknowledge support from the Austrian-French AMADEUS project 35598VB - ChargeDisq.

\bibliographystyle{abbrv}

\begin{thebibliography}{10}
\bibitem{MR3254605}
B. Adcock, M. Gataric, and A. Hansen.
\newblock On stable reconstructions from nonuniform Fourier measurements.
\newblock {\it SIAM J. Imaging Sci.}, 7(3):1690--1723, 2014.

\bibitem{adgaro16}
B. Adcock, M. Gataric, and J. L. Romero.
\newblock Computing reconstructions from nonuniform Fourier samples:
Universality of stability barriers and stable sampling rates.
\newblock {\it Appl. Comput. Harmon. Anal.}, 2017.

\bibitem{MR3615409}
B. Adcock, A. C. Hansen, C. Poon, and B. Roman.
\newblock Breaking the coherence barrier: a new theory for compressed sensing.
\newblock {\it Forum Math. Sigma}, 5:e4, 84, 2017.

\bibitem{MR3382105}
B. Adcock, A. C. Hansen, and B. Roman.
\newblock The quest for optimal sampling: computationally efficient,
structure-exploiting measurements for compressed sensing.
\newblock In {\it Compressed sensing and its applications}, Appl. Numer.
Harmon. Anal., pages 143--167. Birkhäuser/Springer, Cham, 2015.

\bibitem{MR2224392}
R. Balan, P. G. Casazza, C. Heil, and Z. Landau.
\newblock Density, overcompleteness, and localization of frames. I. Theory.
\newblock {\it J. Fourier Anal. Appl.}, 12(2):105--143, 2006.

\bibitem{BW}
J. J. Benedetto and H. C. Wu.
\newblock Nonuniform sampling and spiral {M}{R}{I} reconstruction.
\newblock In {\it Wavelet Applications in Signal and Image Processing VIII},
volume 4119, pages 130--142. International Society for Optics and Photonics,
2000.

\bibitem{B1}
A. Beurling.
\newblock Local harmonic analysis with some applications to differential
operators.
\newblock In {\it Some Recent Advances in the Basic Sciences, Vol. 1 (Proc. Annual Sci. Conf., Belfer Grad. School Sci., Yeshiva Univ., New York, 1962--1964)}, pages 109--125. Belfer Graduate School of Science, Yeshiva Univ., New York, 1966.

\bibitem{B2}
A. Beurling.
\newblock {\it The collected works of Arne Beurling. Vol. 2}.
\newblock Contemporary Mathematicians. Birkhäuser Boston, Inc., Boston, MA,
1989.
\newblock Harmonic analysis, Edited by L. Carleson, P. Malliavin, J. Neuberger
and J. Wermer.

\bibitem{MR3480103}
J. Bigot, C. Boyer, and P. Weiss.
\newblock An analysis of block sampling strategies in compressed sensing.
\newblock {\it IEEE Trans. Inform. Theory}, 62(4):2125--2139, 2016.

\bibitem{boyer2017compressed}
C. Boyer, J. Bigot, and P. Weiss.
\newblock Compressed sensing with structured sparsity and structured acquisition.
\newblock {\it Appl. Comput. Harmon. Anal.}, 46(2):312--350, 2019. 

\bibitem{MR3579696}
C. Boyer, N. Chauffert, P. Ciuciu, J. Kahn, and P. Weiss.
\newblock On the generation of sampling schemes for magnetic resonance imaging.
\newblock {\it SIAM J. Imaging Sci.}, 9(4):2039--2072, 2016.

\bibitem{MR2228740}
E. J. Cand\`es and J. Romberg.
\newblock Quantitative robust uncertainty principles and optimally sparse decompositions.
\newblock {\it Found. Comput. Math.}, 6(2):227--254, 2006.

\bibitem{MR3267158}
N. Chauffert, P. Ciuciu, J. Kahn, and P. Weiss.
\newblock Variable density sampling with continuous trajectories.
\newblock {\it SIAM J. Imaging Sci.}, 7(4):1962--1992, 2014.

\bibitem{MR1738406}
A. Cohen, R. DeVore, P. Petrushev, and H. Xu.
\newblock Nonlinear approximation and the space {$\BV(\R^2)$}.
\newblock {\it Amer. J. Math.}, 121(3):587--628, 1999.

\bibitem{MR1261635}
R. A. DeVore and G. G. Lorentz.
\newblock {\it Constructive approximation}, volume 303 of {\it Grundlehren der
Mathematischen Wissenschaften [Fundamental Principles of Mathematical
Sciences]}.
\newblock Springer-Verlag, Berlin, 1993.

\bibitem{MR2241189}
D. L. Donoho.
\newblock Compressed sensing.
\newblock {\it IEEE Trans. Inform. Theory}, 52(4):1289--1306, 2006.

\bibitem{DS}
R. J. Duffin and A. C. Schaeffer.
\newblock A class of nonharmonic Fourier series.
\newblock {\it Trans. Amer. Math. Soc.}, 72:341--366, 1952.

\bibitem{MR3742438}
H. Führ, K. Gröchenig, A. Haimi, A. Klotz, and J. L. Romero.
\newblock Density of sampling and interpolation in reproducing kernel Hilbert spaces.
\newblock {\it J. Lond. Math. Soc. (2)}, 96(3):663--686, 2017.

\bibitem{MR3050313}
J.-P. Gabardo.
\newblock Convolution inequalities for positive {B}orel measures on $\R^d$ and Beurling density.
\newblock In {\it Excursions in harmonic analysis. {V}olume 2}, Appl. Numer.
Harmon. Anal., pages 23--47. Birkhäuser/Springer, New York, 2013.

\bibitem{MR3796637}
J.-P. Gabardo.
\newblock Weighted convolution inequalities and {B}eurling density.
\newblock In {\it Frames and harmonic analysis}, volume 706 of {\it Contemp. Math.}, pages 175--200. Amer. Math. Soc., Providence, RI, 2018.

\bibitem{MR3209727}
A. Gelb and G. Song.
\newblock A frame theoretic approach to the nonuniform fast Fourier transform.
\newblock {\it SIAM J. Numer. Anal.}, 52(3):1222--1242, 2014.

\bibitem{MR1613711}
K. Gröchenig.
\newblock Irregular sampling, Toeplitz matrices, and the approximation of entire functions of exponential type.
\newblock {\it Math. Comp.}, 68(226):749--765, 1999.

\bibitem{grorro15}
K. Gröchenig, J. Ortega-Cerdà, and J. L. Romero.
\newblock Deformation of Gabor systems.
\newblock {\it Adv. Math.}, 277:388--425, 2015.

\bibitem{gruv15}
K. Gröchenig, J. L. Romero, J. Unnikrishnan, and M. Vetterli.
\newblock On minimal trajectories for mobile sampling of bandlimited fields.
\newblock {\it Appl. Comput. Harmon. Anal.}, 39(3):487--510, 2015.

\bibitem{MR2231911}
A. Iosevich and M. N. Kolountzakis.
\newblock A Weyl type formula for Fourier spectra and frames.
\newblock {\it Proc. Amer. Math. Soc.}, 134(11):3267--3274, 2006.

\bibitem{MR1744572}
A. Iosevich and S. Pedersen.
\newblock How large are the spectral gaps?
\newblock {\it Pacific J. Math.}, 192(2):307--314, 2000.

\bibitem{MR3528398}
P. Jaming and E. Malinnikova.
\newblock An uncertainty principle and sampling inequalities in Besov spaces.
\newblock {\it J. Fourier Anal. Appl.}, 22(4):768--786, 2016.

\bibitem{landau1967sampling}
H. Landau.
\newblock Sampling, data transmission, and the nyquist rate.
\newblock {\it Proceedings of the IEEE}, 55(10):1701--1706, 1967.

\bibitem{MR0222554}
H. J. Landau.
\newblock Necessary density conditions for sampling and interpolation of certain entire functions.
\newblock {\it Acta Math.}, 117:37--52, 1967.

\bibitem{lustig2008compressed}
M. Lustig, D. L. Donoho, J. M. Santos, and J. M. Pauly.
\newblock Compressed sensing MRI.
\newblock {\it IEEE signal processing magazine}, 25(2):72--82, 2008.

\bibitem{MR3223467}
Y. Lyubarskii and J. Ortega-Cerdà.
\newblock Bandlimited Lipschitz functions.
\newblock {\it Appl. Comput. Harmon. Anal.}, 37(2):307--324, 2014.

\bibitem{MR2929058}
S. Nitzan and A. Olevskii.
\newblock Revisiting Landau's density theorems for Paley-Wiener spaces.
\newblock {\it C. R. Math. Acad. Sci. Paris}, 350(9-10):509--512, 2012.

\bibitem{OU}
A. Olevskii and A. Ulanovskii.
\newblock On multi-dimensional sampling and interpolation.
\newblock {\it Anal. Math. Phys.}, 2(2):149--170, 2012.

\bibitem{MR1676043}
J. Ortega-Cerdà.
\newblock Sampling measures.
\newblock {\it Publ. Mat.}, 42(2):559--566, 1998.

\bibitem{MR3398944}
R. B. Platte, A. J. Gutierrez, and A. Gelb.
\newblock Fourier reconstruction of univariate piecewise-smooth functions from non-uniform spectral data with exponential convergence rates.
\newblock {\it Appl. Comput. Harmon. Anal.}, 39(3):427--449, 2015.

\bibitem{MR2040080}
K. Seip.
\newblock {\it Interpolation and sampling in spaces of analytic functions}, volume 33 of {\it University Lecture Series}.
\newblock American Mathematical Society, Providence, RI, 2004.

\bibitem{singh2006active}
A. Singh, R. Nowak, and P. Ramanathan.
\newblock Active learning for adaptive mobile sensing networks.
\newblock In {\it Proceedings of the 5th international conference on Information processing in sensor networks}, pages 60--68. ACM, 2006.

\bibitem{MR2231856}
T. Strohmer and J. Tanner.
\newblock Fast reconstruction methods for bandlimited functions from periodic nonuniform sampling.
\newblock {\it SIAM J. Numer. Anal.}, 44(3):1073--1094, 2006.

\bibitem{MR2286012}
Q. Sun.
\newblock Nonuniform average sampling and reconstruction of signals with finite rate of innovation.
\newblock {\it SIAM J. Math. Anal.}, 38(5):1389--1422, 2006/07.

\bibitem{uv2}
J. Unnikrishnan and M. Vetterli.
\newblock Sampling and reconstruction of spatial fields using mobile sensors.
\newblock {\it IEEE Trans. Signal Process.}, 61(9):2328--2340, 2013.

\bibitem{uv1}
J. Unnikrishnan and M. Vetterli.
\newblock Sampling high-dimensional bandlimited fields on low-dimensional manifolds.
\newblock {\it IEEE Trans. Inform. Theory}, 59(4):2103--2127, 2013.

\bibitem{unser2000sampling}
M. Unser.
\newblock Sampling—50 years after shannon.
\newblock {\it Proceedings of the IEEE}, 88(LIB-ARTICLE-2000-002):569--587, 2000.

\bibitem{MR2679810}
A. Viswanathan, A. Gelb, D. Cochran, and R. Renaut.
\newblock On reconstruction from non-uniform spectral data.
\newblock {\it J. Sci. Comput.}, 45(1-3):487--513, 2010.

\bibitem{MR591684}
R. M. Young.
\newblock {\it An introduction to nonharmonic Fourier series}, volume 93 of {\it Pure and Applied Mathematics}.
\newblock Academic Press, Inc. [Harcourt Brace Jovanovich, Publishers], New York-London, 1980.
\end{thebibliography}

\end{document}